\documentclass[11pt]{article}
\usepackage{amsfonts}
\usepackage{latexsym,amssymb,amsmath,graphics,cite,arydshln}
\usepackage{tikz}
\usepackage{fancyhdr}
\usepackage{lipsum}
\usepackage{lineno}
\usepackage{color}

\usepackage{algorithm} 
\usepackage{algorithmic} 
\usepackage{multirow} 
\usepackage{amsmath}
\usepackage{xcolor}

\begin{document}
\newcommand{\qed}{\hphantom{.}\hfill $\Box$\medbreak}
\newcommand{\proof}{\noindent{\bf Proof \ }}

\newtheorem{theorem}{Theorem}[section]
\renewcommand{\theequation}{\thesection.\arabic{equation}}
\newtheorem{lemma}[theorem]{Lemma}
\newtheorem{corollary}[theorem]{Corollary}
\newtheorem{remark}[theorem]{Remark}
\newtheorem{example}[theorem]{Example}
\newtheorem{definition}[theorem]{Definition}
\newtheorem{construction}[theorem]{Construction}
\newtheorem{fact}[theorem]{Fact}
\newtheorem{proposition}[theorem]{Proposition}

\begin{center}
{\Large\bf Parallel multilevel constructions for constant dimension codes  \footnote{Supported by NSFC under Grant 11971053 (Y. Chang), and NSFC under Grant 11871095 (T. Feng).}
}

\vskip12pt

Shuangqing Liu, Yanxun Chang, Tao Feng\\[2ex] {\footnotesize Department of Mathematics, Beijing Jiaotong University, Beijing 100044, P. R. China}\\
{\footnotesize
16118420@bjtu.edu.cn, yxchang@bjtu.edu.cn, tfeng@bjtu.edu.cn}
\vskip12pt

\end{center}

\vskip12pt

\noindent {\bf Abstract:} Constant dimension codes (CDCs), as special subspace codes, have received a lot of attention due to their application in random network coding. This paper introduces a family of new codes, called rank metric codes with given ranks (GRMCs), to generalize the parallel construction in [Xu and Chen, IEEE Trans. Inf. Theory, 64 (2018), 6315--6319] and the classic multilevel construction. A Singleton-like upper bound and a lower bound for GRMCs derived from Gabidulin codes are given. Via GRMCs, two effective constructions for CDCs are presented by combining the parallel construction and the multilevel construction. Many CDCs with larger size than the previously best known codes are given. The ratio between the new lower bound and the known upper bound for $(4\delta,2\delta,2\delta)_q$-CDCs is calculated. It is greater than 0.99926 for any prime power $q$ and any $\delta\geq 3$.
\vskip12pt

\noindent {\bf Keywords}: constant dimension code, rank-metric code, multilevel construction, parallel construction.


\section{Introduction}



Subspace codes, constant dimension codes in particular, have drawn significant attention due to the work by K\"{o}tter and Kschischang \cite{kk}, where they presented an application of such codes for error correction in random network coding.

Let $\mathbb F_q$ be the finite field of order $q$, and $\mathbb F_q^n$ be the set of all vectors of length $n$ over $\mathbb F_q$. $\mathbb F_q^n$ is an $n$-dimensional vector space over $\mathbb F_q$. Given a nonnegative integer $k\leq n$, the set of all $k$-dimensional subspaces of $\mathbb F_q^n$ is called the {\em Grassmannian} ${\cal G}_q(n,k)$. The cardinality of ${\cal G}_q(n,k)$ is given by the {\em $q$-ary Gaussian coefficient}
$$
|{\cal G}_q(n,k)|=
\begin{bmatrix}
 n \\
 k \\
\end{bmatrix}_q\triangleq \prod \limits_{i=0}^{k-1}\frac{q^{n-i}-1}{q^{k-i}-1}.
$$
For any two subspaces $\mathcal U,\mathcal V\in \mathcal {\cal G}_q(n,k)$, their {\em subspace distance} is defined by
\begin{eqnarray}\label{1.1}
d_S(\mathcal U,\mathcal V)\triangleq \dim \mathcal U+\dim \mathcal V-2\dim(\mathcal U\cap \mathcal V)=2(k-\dim(\mathcal U\cap\mathcal V)).
\end{eqnarray}

A subset $\cal C$ of the Grassmannian ${\cal G}_q(n,k)$ is called an $(n,d,k)_q$ {\em constant-dimension code} (CDC), if $d_S(\mathcal U,\mathcal V)\geq d$ for all $\mathcal U, \mathcal V\in \mathcal C$ and $\mathcal U\neq \mathcal V$. Elements in $\cal C$ are called {\em codewords}.
An $(n,d,k)_q$-CDC with $M$ codewords is written as an $(n,M,d,k)_q$-CDC. Given $n,d,k$ and $q$, denote by $A_q(n, d, k)$ the maximum number of codewords among all $(n,d,k)_q$-CDCs. An $(n,d,k)_q$-CDC with
$A_q(n, d, k)$ codewords is said to be {\em optimal}.

Without loss of generality, assume that $n\geq 2k$. This assumption can be made as a consequence of the fact $A_q(n, d, k)=A_q(n, d, n-k)$, which can be obtained by taking orthogonal complements of subspaces (cf. \cite{xf}). Furthermore, by \eqref{1.1}, when $2k<d$, any nonempty $(n,d,k)_q$-CDC consists of exactly one codeword. Therefore, we always assume that $n\geq 2k\geq d$.

A $k$-dimensional subspace $\mathcal U$ of $\mathbb F^n_q$ can be represented by a {\em $k\times n$ generator matrix} $\boldsymbol{U}$ whose rows form a basis of $\mathcal U$. Let $\mathbb F^{k \times n}_q$ denote the set of all $k \times n$ matrices over $\mathbb F_q$. For $\mathcal U,\mathcal V\in {\cal G}_q(n,k)$, the subspace distance on ${\cal G}_q(n,k)$ is also given by
\begin{eqnarray}\label{1.2}
d_S(\mathcal U,\mathcal V)=2\cdot {\rm rank} \binom{\boldsymbol{U}}{\boldsymbol{V}}-2k,
\end{eqnarray}
where $\boldsymbol{U},\boldsymbol{V} \in \mathbb F^{k\times n}_q$ are matrices such that $\mathcal U=$ rowspace$(\boldsymbol{U})$ and $\mathcal V=$ rowspace$(\boldsymbol{V})$. The two matrices are usually not unique.

By \eqref{1.2}, in a CDC, the minimum subspace distance $d$ is even. If $d=0$ or $2$, then all $k$-dimensional subspaces of $\mathbb F_q^n$ constitute an optimal $(n,d,k)_q$-CDC. It follows that $d=4$ is the minimum nontrivial value.

\subsection{Lifted maximum rank distance codes}

To obtain optimal CDCs, Silva, Kschischang and K\"{o}tter \cite{skk} pointed out that lifted maximum rank distance (MRD) codes can result in asymptotically optimal CDCs, and can be decoded efficiently in the context of random linear network coding.

For a matrix $\boldsymbol{A} \in \mathbb F^{m \times n}_q$, the rank of $\boldsymbol{A}$ is denoted by rank$(\boldsymbol{A})$. The set $\mathbb F^{m \times n}_q$ is an $\mathbb F_q$-vector space. The {\em rank distance} on $\mathbb F^{m \times n}_q$ is defined by
\begin{center}
  $d_R(\boldsymbol{A}, \boldsymbol{B})={\rm rank}(\boldsymbol{A}-\boldsymbol{B}),~{\rm for}~\boldsymbol{A}, \boldsymbol{B} \in \mathbb F^{m \times n}_q$.
\end{center}
An $[m \times n, k, \delta]_q$ {\em rank-metric code} $\mathcal D$ is a $k$-dimensional $\mathbb F_q$-linear subspace of $\mathbb F^{m \times n}_q$ with {\em minimum rank distance} $$\delta=\underset{\boldsymbol{A,B} \in \mathcal D, \boldsymbol{A}\neq \boldsymbol{B}}{\min}\{d_R(\boldsymbol{A},\boldsymbol{B})\}.$$
Clearly $$\delta=\underset{\boldsymbol{A} \in \mathcal D, \boldsymbol{A}\neq \textbf{0}}{\min}\{{\rm rank}(\boldsymbol{A})\}.$$
Elements in $\cal D$ are called {\em codewords}. The Singleton-like upper bound for rank-metric codes implies that $$k \leq \max\{m,n\}(\min\{m,n\}-\delta +1)$$
holds for any $[m \times n, k, \delta]_q$ code. When the equality holds, $\cal D$ is called a {\em linear maximum rank distance code}, denoted by an MRD$[m \times n, \delta]_q$ code. Linear MRD codes exists for all feasible parameters (cf. \cite{d,g,r}).

Write $\boldsymbol{I}_{k}$ as the $k\times k$ identity matrix.

\begin{proposition}[Lifted MRD codes,\cite{skk}] \label{lift}
Let $n\geq 2k$. The lifted MRD code $$\mathcal C=\{{\rm rowspace}(\boldsymbol{I}_k \mid \boldsymbol{A}):\boldsymbol{A} \in \mathcal D\}$$ is an $(n,q^{(n-k)(k-\delta+1)},2\delta,k)_q$-CDC, where $\mathcal D$ is an MRD$[k\times (n-k),\delta]_q$ code.
\end{proposition}

We outline the proof of Proposition \ref{lift} for later use. It suffices to check the subspace distance of $\mathcal C$. For any $\mathcal U, \mathcal V\in \mathcal C$ and $\mathcal U\neq\mathcal V$, where $\mathcal U=$ rowspace$(\boldsymbol{I}_k\mid \boldsymbol{A})$ and $\mathcal V=$ rowspace$(\boldsymbol{I}_k\mid \boldsymbol{B})$, we have
\begin{align*}
d_S(\mathcal U,\mathcal V)&
=2\cdot{\rm rank}\left(
                                                  \begin{array}{ccc}
                                                    \boldsymbol I_k & \boldsymbol A \\
                                                    \boldsymbol I_k & \boldsymbol B \\
                                                  \end{array}
                                                \right)-2k
                                                  =2\cdot{\rm rank}\left(
                                                  \begin{array}{ccc}
                                                    \boldsymbol I_k & \boldsymbol A \\
                                                    \boldsymbol O & \boldsymbol {B-A} \\
                                                  \end{array}
                                                \right)-2k\\
                                                &=2\cdot{\rm rank}(\boldsymbol {B-A})\geq 2\delta.
\end{align*}


Given $n,\delta,k$ and $q$, denote by $\bar{A}_q(n, 2\delta, k)$ the maximum number of codewords among all $(n,2\delta,k)_q$-CDC containing a lifted MRD code $(n,q^{(n-k)(k-\delta+1)},2\delta,k)_q$-CDC as a subset. Many constructions for CDCs with large number of codewords known in the literature (cf. \cite{es09,gy,se13,st15,skk,s,tr}) produce codes containing a lifted MRD code. This motivates the study, initialized by Etzion and Silberstein \cite{es13}, on determining the lower and upper bounds on $\bar{A}_q(n, d, k)$.

Etzion and Silberstein \cite{es09} presented a simple but effective construction, named the multilevel construction, which generalizes the lifted MRD codes.
Trautmann and  Rosenthal in \cite{tr} improved the multilevel construction by pending dots. Using the idea of pending dots and graph matchings, Etzion, Silberstein in \cite{es13} and Silberstein, Trautmann in \cite{st15} constructed large subspace codes in ${\cal G}_q(n,k)$ of minimum subspace distance $d=4$ or $2k-2$. Xu and Chen \cite{xc} presented a new construction which can be also seen as a generalization of the lifted MRD codes. Heinlein \cite{Heinlein} summarized the upper bounds of CDCs which contain lifted MRD codes as follows.

\begin{theorem} {\rm \cite[Theorem 1]{Heinlein}} \label{upper-h}
For $n\geq 2k$, let $\mathcal C$ be an $(n,2\delta,k)_q$-CDC which contains a lifted MRD code.
\begin{itemize}
\item[$(1)$] If $k<2\delta$ and $n\geq 3\delta$, then $\bar{A}_q(n,2\delta,k)\leq q^{(n-k)(k-\delta+1)}+A_q(n-k,2(2\delta-k),\delta)$.
If additionally $k=\delta$, $n\equiv r \pmod k$, $0\leq r< k$, and ${\small\begin{bmatrix}
 r \\
 1 \\
 \end{bmatrix}_q}<k$, or $(n,2\delta,k)\in\left\{(6+3l,4+2l,3+l),(6l,4l,3l)\mid l\geq1\right\}$, then the bound can be achieved.

\item[$(2)$] If $k<2\delta$ and $n<3\delta$, then $\bar{A}_q(n,2\delta,k)= q^{(n-k)(k-\delta+1)}+1$.

\item[$(3)$] If $2\delta\leq k<3\delta$, then
\begin{align*}
\bar{A}_q(n,2\delta,k)\leq q^{(n-k)(k-\delta+1)}&+A_q(n-k,6\delta-2k,2\delta)\\&+q^{(k-2\delta+1)(n-k-\delta)}\frac{{\small\begin{bmatrix}
 n-k \\
 \delta \\
 \end{bmatrix}_q}{\small\begin{bmatrix}
 k \\
 2\delta-1 \\
 \end{bmatrix}_q}}{{\small\begin{bmatrix}
 k-\delta \\
 \delta-1 \\
 \end{bmatrix}_q}}.\end{align*}
\end{itemize}
\end{theorem}
For more information on constructions and bounds for subspace codes, the interested reader is refered to \cite{aa,es13,fw,kk,hkkw,ev11,kku,se131,gr,xf,wxs,tmbr,chwx,gt}.

\subsection{Our contribution}

This paper is devoted to constructing large constant dimension codes which contain a lifted MRD code as a subset.

Section 2 generalizes a construction for CDCs in \cite{xc} by introducing a family of new codes, called rank metric codes with given ranks (GRMCs). This generalized construction is called a parallel construction (see Construction \ref{new11}). We shall establish a Singleton-like upper bound (see Proposition \ref{grmc-upper}) and a lower bound (see Proposition \ref{grmc}) for GRMCs by using Gabidulin codes. Very recently, Heinlein \cite{Heinlein-1} also introduced a similar concept to GRMCs. He presented several lower bounds for GRMCs, but most focus on special parameters. Here our construction is for general parameters. Applying Construction \ref{new11} together with Proposition \ref{grmc}, we give a lower bound on $\bar{A}_q(n,2\delta,k)$ for any $n\geq2k>2\delta>0$.

Section 3 presents two effective constructions for CDCs (see Constructions \ref{comb} and \ref{new}) by combining the parallel construction and the classic multilevel construction. Constructions \ref{comb} shows that if a multilevel construction satisfies the weight of the first $n-k$ positions of every identifying vector is no less than $\delta$, then the multilevel construction can be combined with a parallel construction. Constructions \ref{new} shows that if identifying vectors in a multilevel construction dissatisfy the condition in Constructions \ref{comb}, the multilevel construction is still possible to be combined with a parallel construction. In both construction, GRMCs play an important role.


In principle, people can always pick up suitable identifying vectors and then use the classic multilevel construction to construct optimal CDC. However, how to choose identifying vectors effectively is still an open and different problem. The combination of the parallel construction and the multilevel construction helps to weaken the requirement for identifying vectors and provides good constant dimension codes with large size. Applying Constructions \ref{comb} and \ref{new}, we establish new lower bounds for CDCs (see Theorems \ref{new-3}, \ref{con4}, \ref{cdc45} and Corollary \ref{cor2}). Many CDCs with larger size than the previously best known codes in \cite{hkkw} are given (see Appendix B). We also calculate the ratio between our lower bound and the known upper bound for $(4\delta,2\delta,2\delta)_q$-CDCs. It is greater than 0.99926 for any prime power $q$ and any $\delta\geq 3$ (see Remark \ref{4delta}).

\section{Parallel construction}

In \cite{xc}, Xu and Chen presented an interesting construction to establish new lower bounds for $A_q(2k,2\delta,k)$, where $k\geq 2\delta$.

\begin{theorem}{\rm \cite[Theorem 3 and Corollary 4]{xc}}\label{xc1}
For any positive integers $k$ and $\delta$ such that $k\geq 2\delta$,
\begin{center}
$A_q(2k,2\delta,k)\geq q^{k(k-\delta+1)}+\sum_{i=\delta}^{k-\delta}A_i$,
\end{center}
where $A_i$ denotes the number of codewords with rank $i$ in an MRD$[k\times k,\delta]_q$ code.
\end{theorem}

Theorem \ref{xc1} relies on the use of the rank distribution of an MRD code.

\begin{theorem} [Rank distribution \cite{d,g}] \label{distribution}
Let $m\ge n$. Let $\mathcal D$ be an MRD$[m\times n ,\delta]_q$ code, and $A_i = |\{M \in \mathcal D: {\rm rank}(M) = i\}|$ for $0\leq i \leq n$. Its rank distribution is given by $A_0 = 1$, $A_i = 0$ for $1\leq i \leq\delta-1$, and
\begin{center}
$A_{\delta+i}=\begin{bmatrix}
 n \\
 \delta+i \\
 \end{bmatrix}_q\sum_{j=0}^{i}(-1)^{j-i}\begin{bmatrix}
 \delta+i \\
 i-j \\
 \end{bmatrix}_q q^{\binom {i-j} {2}}(q^{m(j+1)}-1)$
\end{center}
 for $0 \leq i\leq n-\delta$.
\end{theorem}

In this section, we shall generalize Theorem \ref{xc1} by introducing the concept of rank metric codes with given ranks (GRMCs), which can be seen as a generalization of constant-rank codes. Constant-rank codes have been discussed systematically in \cite{gy}.

We remark that very recently, based on Theorem \ref{xc1}, Chen et.al \cite{chwx} and Heinlein \cite{Heinlein-1} generalized linkage constructions to establish some lower bounds of CDCs independently.

\subsection{Rank metric codes with given ranks}

Let $K\subseteq \{0,1,\ldots,n\}$ and $\delta$ be a positive integer. We say $\mathcal D\subseteq\mathbb F_q^{m\times n}$ is an $(m\times n,\delta,K)_q$ {\em rank metric code with given ranks} (GRMC) if it satisfies
\begin{itemize}
\item[$(1)$] ${\rm rank}(\boldsymbol{D})\in K$ for any $\boldsymbol{D} \in \mathcal D$;
\item[$(2)$] $d_R(\boldsymbol{D}_1,\boldsymbol{D}_2)={\rm rank}(\boldsymbol{D}_1-\boldsymbol{D}_2)\geq \delta$ for any $\boldsymbol{D}_1,\boldsymbol{D}_2 \in \mathcal D$ and $\boldsymbol{D}_1\neq \boldsymbol{D}_2$.
\end{itemize}
If $|\mathcal D|=M$, then it is often written as an $(m\times n, M,\delta,K)_q$-GRMC. Given $m$, $n$, $K$ and $\delta$, denote by $A^R_q(m\times n, \delta, K)$ the maximum number of codewords among all $(m\times n,\delta,K)_q$-GRMCs.

When $K=\{0,1,\ldots,n\}$, a GRMC is just a usual rank-metric code (not necessarily linear). When $K=\{t\}$ for $0\leq t\leq n$, a GRMC is often called a {\em constant-rank code} (cf. \cite{gy}). Consequently, when $m\geq n$,
\begin{center}
$ \max\limits_{t\in K}A^R_q(m\times n, \delta, t)\leq A^R_q(m\times n, \delta, K)\leq q^{m(n-\delta+1)}$.
\end{center}
Usually, $K$ is selected as a set of consecutive integers. Let $[t_1, t_2]$ denote the set of integers $k$ such that $t_1\leq k \leq t_2$.

\begin{proposition}[Singleton-like upper bound]\label{grmc-upper}
For all $0\leq i,j\leq\min\{\delta-1,t_1\}$,
\begin{center}
$A^R_q(m\times n, \delta, [t_1,t_2])\leq A^R_q((m-i)\times (n-j), \delta-l, [t_1-l,t])$,
\end{center}
where $l=\max\{i,j\}$ and $t=\min\{m-i,n-j,t_2\}$.
\end{proposition}

\begin{proof}
Let $\mathcal D$ be any $(m\times n,M,\delta,[t_1,t_2])_q$-GRMC. For any $0\leq i,j\leq\min\{\delta-1,t_1\}$, the $\mathcal D_{ij}$ is obtained by removing the same $i$ rows and $j$ columns from every codeword in $\mathcal D$. Then $\mathcal D_{ij}$ is an $((m-i)\times (n-j),M,\delta-l,[t_1-l,t])_q$-GRMC. \qed
\end{proof}

When $t_1=t_2$, Proposition \ref{grmc-upper} provides an upper bound for constant-rank codes (see also \cite[Proposition 7]{gy}).

We remark that very recently, Heinlein \cite{Heinlein-1} also introduced the concept of GRMCs for the case $K=[0,t_2]$. He established a similar upper bound for GRMCs to that in Proposition \ref{grmc-upper}. He also presented several lower bounds for GRMCs, but most focus on special parameters. Here we shall give a general construction on GRMCs.

To construct GRMCs, we need a special class of MRD codes, named Gabidulin codes. Let \boldmath $\mathbf{\beta}$ \unboldmath $=(\beta_0,  \beta_1, \ldots, \beta_{m-1})$ be an ordered basis of $\mathbb F_{q^m}$ over $\mathbb F_q$. There is a natural bijective map $\Psi_m$ from $\mathbb F_{q^m}^n$
to $\mathbb F_{q}^{m \times n}$ as follows:
\begin{center}
  $\Psi_m: \mathbb F^n_{q^m} \longrightarrow \mathbb F^{m\times n}_q$
\end{center}
\begin{center}
  $\textbf{a}=(a_0, a_1, \ldots, a_{n-1}) \longmapsto \boldsymbol{A},
  $
\end{center}
where $\boldsymbol{A}=\Psi_m(\textbf{a})\in \mathbb F^{m\times n}_q$ is defined such that
\begin{equation*}
 a_j=\sum_{i=0}^{m-1} A_{i, j}\beta_i
\end{equation*}
for any $j\in [n]$. The map $\Psi_m$ will be used to facilitate switching between a vector in $\mathbb F_{q^m}$ and its matrix representation over $\mathbb F_q$.  In the sequel, we use both representations, depending on what is more convenient in the context.

For any positive integer $i$ and any $a\in \mathbb F_{q^m}$, set $a^{[i]} \triangleq a^{q^i}$. Let $m\geq n$ and $\delta$ be a positive integer. A {\em Gabidulin code} $\mathcal{G}[m\times n,\delta]_q$ is an MRD$[m\times n,\delta]_q$ code whose generator matrix $\boldsymbol{G}$ in vector representation is
\renewcommand\theequation{\arabic{section}.\arabic{equation}}
\begin{equation*}\label{gm}
\boldsymbol{G}=\left(
               \begin{array}{cccc}
                 g_0 & {g_1} & \cdots & g_{n-1} \\
                 g_0^{[1]} & g_1^{[1]} & \cdots & g_{n-1}^{[1]} \\
                 \vdots & \vdots & \ddots & \vdots \\
                 g_0^{[n-\delta]} & g_1^{[n-\delta]} & \cdots & g_{n-1}^{[n-\delta]} \\
               \end{array}
             \right),
\end{equation*}
where $g_0, g_1, \ldots, g_{n-1} \in \mathbb F_{q^m}$ are linearly independent over $\mathbb F_q$ (see \cite{g}). Then $\mathcal{G}[m\times n,\delta]_q$ can be written as $\{\boldsymbol{uG}:\boldsymbol{u}\in \mathbb F^{n-\delta+1}_{q^m}\}$.

\begin{proposition}[Lower bound]\label{grmc}
Let $m\geq n$ and $1\leq \delta \leq n$. Let $t_1$ be a nonnegative integer and $t_2$ be a positive integer such that $t_1\leq t_2\leq n$. Then
\begin{equation}\label{}\nonumber
{A^R_q(m\times n, \delta, [t_1,t_2])\geq \left \{
\begin {aligned}
&\sum_{i=t_1}^{t_2}A_i(\delta),&t_2\geq \delta;\\
&\max\limits_{\max\{1,t_1\}\leq a <\delta}\{\lceil\frac{\sum_{i=\max\{1,t_1\}}^{t_2}A_i(a)}{q^{m(\delta-a)}-1}\rceil\},&t_2<\delta,\\
\end {aligned}
\right.}
\end{equation}
where $A_i(x)$ denotes the number of codewords with rank $i$ in an MRD$[m\times n,x]_q$ code.
\end{proposition}

\begin{proof}
When $t_2\geq \delta$, all codewords with ranks from $[t_1,t_2]$ in an MRD$[m\times n,\delta]_q$ code form an $(m\times n,\delta,[t_1,t_2])_q$-GRMC with $\sum_{i=t_1}^{t_2}A_i(\delta)$ codewords. Thus $A^R_q(m\times n, \delta, [t_1,t_2])\geq \sum_{i=t_1}^{t_2}A_i(\delta)$ for any $t_2\geq \delta$.

When $t_2<\delta$, take any integer $a$ such that $\max\{1,t_1\}\leq a<\delta$. Let $\mathcal D_1$ be a $\mathcal{G}[m\times n,\delta]_q$ code whose generator matrix $\boldsymbol{G}_1$ in vector representation is
\renewcommand\theequation{\arabic{section}.\arabic{equation}}
\begin{center}
$\boldsymbol{G}_1=\left(
               \begin{array}{cccc}
                 g_0 & {g_1} & \cdots & g_{n-1} \\
                 g_0^{[1]} & g_1^{[1]} & \cdots & g_{n-1}^{[1]} \\
                 \vdots & \vdots & \ddots & \vdots \\
                 g_0^{[n-\delta]} & g_1^{[n-\delta]} & \cdots & g_{n-1}^{[n-\delta]} \\
               \end{array}
             \right)$,
\end{center}
where $g_0, g_1, \ldots, g_{n-1} \in \mathbb F_{q^m}$ are linearly independent over $\mathbb F_q$.
Let $\mathcal D_2$ be a $\mathcal{G}[m\times n,n-\delta+a+1]_q$ code whose generator matrix $\boldsymbol{G}_2$ in vector representation is
\begin{center}
$\boldsymbol{G}_2=\left(
               \begin{array}{cccc}
                 g_0^{[n-\delta+1]} & g_1^{[n-\delta+1]} & \cdots & g_{n-1}^{[n-\delta+1]} \\
                 g_0^{[n-\delta+2]} & g_1^{[n-\delta+2]} & \cdots & g_{n-1}^{[n-\delta+2]} \\
                 \vdots & \vdots & \ddots & \vdots \\
                 g_0^{[n-a]} & g_1^{[n-a]} & \cdots & g_{n-1}^{[n-a]} \\
               \end{array}
             \right)$.
\end{center}
Note that $g_0^{[n-\delta+1]},g_1^{[n-\delta+1]},\cdots,g_{n-1}^{[n-\delta+1]}$ are also linearly independent over $\mathbb F_q$. Then $$\bigcup_{\boldsymbol{D}_2\in \mathcal D_2}\bigcup_{\boldsymbol{D}_1\in \mathcal D_1}(\boldsymbol{D}_2+\boldsymbol{D}_1)$$ is a $\mathcal{G}[m\times n,a]_q$ code whose generator matrix is $\left(
\begin{array}{c}
\boldsymbol{G}_1 \\                                                                                                                            \boldsymbol{G}_2 \\                                                                                                                          \end{array}                                                                                                                        \right)$.
For any $\boldsymbol{D}_2\in \mathcal D_2\setminus\{\boldsymbol{O}_{m\times n}\}$, where $\boldsymbol{O}_{m\times n}$ is the $m\times n$ zero matrix, set
$${\cal S}_{\boldsymbol{D}_2}=\{\boldsymbol{D}\in \bigcup_{\boldsymbol{D}_1\in \mathcal D_1}(\boldsymbol{D}_2+\boldsymbol{D}_1)\mid {\rm rank}(\boldsymbol{D})\in [t_1,t_2]\}.$$
For any $\boldsymbol{D}_1,\boldsymbol{D}'_1\in{\cal D}_1$ and $\boldsymbol{D}_1\neq \boldsymbol{D}'_1$, we have ${\rm rank}((\boldsymbol{D}_2+\boldsymbol{D}_1)-(\boldsymbol{D}_2+\boldsymbol{D}'_1))\geq \delta$. It follows that ${\cal S}_{\boldsymbol{D}_2}$ forms an $(m\times n,M_{\boldsymbol{D}_2},\delta,[t_1,t_2])_q$-GRMC, where
by the pigeonhole principle,
$$M_{\boldsymbol{D}_2}\geq \lceil\frac{\sum_{i=\max\{1,t_1\}}^{t_2}A_i(a)}{|\mathcal{D}_2|-1}\rceil=
\lceil\frac{\sum_{i=\max\{1,t_1\}}^{t_2}A_i(a)}{q^{m(\delta-a)}-1}\rceil.$$
Note that due to $\delta>t_2$, if $t_1>0$, then there is no codeword in ${\cal D}_1$ with rank $i\in[t_1,t_2]$. Therefore, $A^R_q(m\times n, \delta, [t_1,t_2])\geq \max\limits_{\max\{1,t_1\}\leq a <\delta}\{\lceil\frac{\sum_{i=\max\{1,t_1\}}^{t_2}A_i(a)}{q^{m(\delta-a)}-1}\rceil\}$. \qed
\end{proof}

\subsection{Generalization of Theorem \ref{xc1}}

\begin{construction}[Parallel construction]\label{new11}
Let $n\geq2k\geq 2\delta$. If there exists a $(k\times (n-k),M,\delta,[0,k-\delta])_q$-GRMC, then there exists an $(n,q^{(n-k)(k-\delta+1)}+M,2\delta,k)_q$-CDC, which contains a lifted MRD code $(n,q^{(n-k)(k-\delta+1)},2\delta,k)_q$-CDC as a subset.
\end{construction}

\begin{proof}
Let $\mathcal D_1$ be an MRD$[k\times (n-k),\delta]_q $ code and $\mathcal D_2$ be a $(k\times (n-k),M,\delta,[0,k-\delta])_q$-GRMC. Let $\mathcal C_1=\{{\rm rowspace}(\boldsymbol I_k\mid \boldsymbol A):\boldsymbol A\in \mathcal D_1\}$ and $\mathcal C_2=\{{\rm rowspace}(\boldsymbol B\mid \boldsymbol I_k): \boldsymbol B \in \mathcal D_2\}$. Then $\mathcal C=\mathcal C_1\cup\mathcal C_2$ forms an $(n,q^{(n-k)(k-\delta+1)}+M,2\delta,k)_q$-CDC.

It suffices to check the subspace distance of $\mathcal C$. For any $\mathcal U=$ rowspace$(\boldsymbol{I}_k\mid \boldsymbol{A})\in \mathcal C_1$ and $\mathcal V=$ rowspace$(\boldsymbol{B}\mid\boldsymbol{I}_k)\in \mathcal C_2$, where
$\boldsymbol A=(\underbrace{\boldsymbol A_1}_{n-2k}|\underbrace{\boldsymbol A_2}_{k})$ and  $\boldsymbol B=(\underbrace{\boldsymbol B_1}_{k}|\underbrace{\boldsymbol B_2}_{n-2k})$, we have
\begin{align*}
d_S(\mathcal U,\mathcal V)&=2\cdot{\rm rank}\left(
                                                  \begin{array}{ccc}
                                                    \boldsymbol I_k & \boldsymbol A_1 & \boldsymbol A_2 \\
                                                    \boldsymbol B_1 & \boldsymbol B_2 & \boldsymbol I_k \\
                                                  \end{array}
                                                \right)-2k\\
                                                &=2\cdot{\rm rank}\left(
                                                  \begin{array}{ccc}
                                                    \boldsymbol I_k & \boldsymbol A_1 & \boldsymbol A_2 \\
                                                    \boldsymbol O & \boldsymbol B_2-\boldsymbol B_1\boldsymbol A_1 & \boldsymbol I_k-\boldsymbol B_1\boldsymbol A_2 \\
                                                  \end{array}
                                                \right)-2k\\
                                                &=2\cdot{\rm rank}(\boldsymbol B_2-\boldsymbol B_1\boldsymbol A_1 \mid  \boldsymbol I_k-\boldsymbol B_1\boldsymbol A_2)\\
                                                & \geq 2\cdot{\rm rank}(\boldsymbol I_k-\boldsymbol B_1\boldsymbol A_2)\\
                                                &\geq 2\cdot{\rm rank}\cdot (\boldsymbol I_k-\boldsymbol B_1\boldsymbol A_2+\boldsymbol B_1\boldsymbol A_2)-2\cdot{\rm rank}(\boldsymbol B_1\boldsymbol A_2)\\
                                                &= 2k-2\cdot{\rm rank}(\boldsymbol B_1\boldsymbol A_2)\\
                                                &\geq 2k-2\cdot{\rm rank}(\boldsymbol B_1)\geq 2k-2\cdot{\rm rank}(\boldsymbol B)\geq 2\delta.
\end{align*}
Given $i\in\{1,2\}$, for any $\mathcal U,\mathcal V \in \mathcal C_i$ and $\mathcal U\neq\mathcal V$, we have $d_S(\mathcal U,\mathcal V)\geq 2\delta$ by the proof of Proposition \ref{lift}. \qed
\end{proof}

Construction \ref{new11} comes from two parallel versions of lifted MRD codes, so it is named as a parallel construction. As a straightforward corollary of Construction \ref{new11}, we have the following result.

\begin{theorem}\label{cor:parallel}
Let $n\geq2k\geq 2\delta$. Then
$$\bar{A}_q(n,2\delta,k)\geq q^{(n-k)(k-\delta+1)}+A^R_q(k\times (n-k), \delta, [0,k-\delta]).$$
\end{theorem}

\begin{theorem}\label{cornew1}
Let $n\geq2k>2\delta>0$. Then
\begin{equation}
\label{}\nonumber
{\bar{A}_q(n,2\delta,k)\geq q^{(n-k)(k-\delta+1)}+ \left \{
\begin {aligned}
&\sum_{i=\delta}^{k-\delta}A_i(\delta)+1,& k\geq 2\delta;\\
&\max\limits_{1\leq a <\delta}\{\lceil\frac{\sum_{i=1}^{k-\delta}A_i(a)}{q^{m(\delta-a)}-1}\rceil\},& k<2\delta,\\
\end {aligned}
\right.}
\end{equation}
where $A_i(x)$ denotes the number of codewords with rank $i$ in an MRD$[m\times n,x]_q$ code.
\end{theorem}

\proof Apply Proposition \ref{grmc} with $t_1=0$ and $t_2=k-\delta$. Then apply Theorem \ref{cor:parallel}. \qed

We remark that if $k\geq 2\delta$, then Theorem \ref{xc1} is a corollary of Theorem \ref{cornew1} by taking $n=2k$.

\section{Parallel multilevel construction}


This section is devoted to presenting two effective constructions for constant dimension codes (Constructions \ref{comb} and \ref{new}) by combining the parallel construction shown in Section 2 and the multilevel construction introduced in \cite{es09}.
First we give a revisit of the multilevel construction in \cite{es09}.

\subsection{Preliminaries}



Etzion and Silberstein \cite{es09} presented the multilevel construction for constant dimension codes by establishing relations between subspace distances and Hamming distances or rank distances (see Lemmas \ref{lem:efc-2} and \ref{lem:efc-1}, respectively, below).
The multilevel construction is a generalization of Proposition \ref{lift} by introducing Ferrers diagram rank-metric codes.

\subsubsection{Ferrers diagram rank-metric codes}

Given positive integers $m$ and $n$, an $m \times n$ {\em Ferrers diagram} $\mathcal F$ is an $m \times n$ array of dots and empty cells such that all dots are shifted to the right of the diagram, the number of dots in each row is less than or equal to the number of dots in the previous row, and the first row has $n$ dots and the rightmost column has $m$ dots.

We always denote by $\gamma_i$, $0\leq i\leq n-1$, the number of dots in the $i$-th column of $\mathcal F$.
Given positive integers $m$ and $n$, and $1\leq \gamma_0 \leq \gamma_1 \leq \cdots \leq \gamma_{n-1}=m$, there exists a unique Ferrers diagram $\mathcal F$ of size $m\times n$ such that the $i$-th column of $\mathcal F$ has cardinality $\gamma_i$ for any $0\leq i\leq n-1$. In this case we write $\mathcal F=[\gamma_0,\gamma_1,\ldots,\gamma_{n-1}]$.

\begin{example}\label{1}
Let $\mathcal F=[2,3,4,5]$, where
\begin{center}
 $\mathcal{F}=\begin{array}{cccc}
                \bullet & \bullet & \bullet & \bullet \\
                \bullet & \bullet & \bullet & \bullet \\
                 & \bullet & \bullet & \bullet \\
                 &  & \bullet & \bullet \\
                 &  &  & \bullet
              \end{array}
 $
\end{center}
be a $5\times 4$ Ferrers diagram.
\end{example}

For a given $m \times n$ Ferrers diagram $\mathcal F$, an $[\mathcal F, k, \delta]_q$ {\em Ferrers diagram rank-metric code} (FDRMC), briefly an $[\mathcal F, k, \delta]_q$ code, is an $[m \times n, k, \delta]_q$ rank-metric codes in which for each $m \times n$ matrix, all entries not in $\mathcal F$ are zero.
If $\mathcal F$ is a {\em full} $m\times n$ diagram with $mn$ dots, then its corresponding FDRM code is just a classical rank-metric code. For a Ferrers diagram $\mathcal F$ of size $m\times n$, one can transpose it to obtain a Ferrers diagram ${\cal F}^{t}$ of size $n\times m$. If there exists an $[\mathcal F, k, \delta]_q$ code, then so does an $[{\mathcal F}^t, k, \delta]_q$ code.  Etzion and Silberstein \cite{es09} established a Singleton-like upper bound on FDRMCs.

\begin{lemma} {\rm \cite[Theorem 1]{es09}} \label{lem:upper bound}
Let $\delta$ be a positive integer. Let $v_i$, $0\leq i\leq \delta-1$, be the number of dots in a Ferrers diagram $\cal F$ which are not contained in the first $i$ rows and the rightmost $\delta-1-i$ columns. Then for any $[\mathcal F, k, \delta]_q$ code, $k\leq \min_{i\in\{0,1,\ldots,\delta-1\}}v_i$.
\end{lemma}

An FDRMC attaining the upper bound in Lemma \ref{lem:upper bound} is called {\em optimal}. Constructions for optimal FDRMCs can be found in \cite{al,lf1,egrw,es09,lf,gr,st15,zg}. We here only quote several known constructions for late use.

\begin{theorem} {\rm \cite[Theorem 3]{egrw}} \label{thm:shortening}
Assume $\mathcal F$ is an $m \times n$ $(m\geq n)$ Ferrers diagram and each of the rightmost $\delta-1$ columns of $\mathcal F$ has at least $n$ dots. Then there exists an optimal $[\mathcal F, \sum_{i=0}^{n-\delta} \gamma_i, \delta]_q$ code for any prime power $q$.
\end{theorem}


\begin{theorem}{\rm \cite[Theorem 9]{egrw}} \label{thm:combine with same dim}
Let $\mathcal F_i$ for $i=1,2$ be an $m_i \times n_i$ Ferrers diagram, and $\mathcal C_i$ be an $[\mathcal F_i, k, \delta_i]_q$ code. Let $\mathcal F_3$ be an $m_3 \times n_3$ full Ferrers diagram with $m_3n_3$ dots, where $m_3 \geq m_1$ and $n_3 \geq n_2$. Let
\begin{center}
$\mathcal F=\left(
  \begin{array}{cc}
    \mathcal F_1 & \mathcal F_3 \\
      & \mathcal F_2 \\
  \end{array}
\right)$
\end{center}
be an $m \times n$ Ferrers diagram $\mathcal F$, where $m=m_2+m_3$ and $n=n_1+n_3$. Then there exists an $[\mathcal F, k, \delta_1+\delta_2]_q$ code ${\cal D}$ such that for any $\textbf{D}\in {\cal D}$, $\textbf{D}|_{\mathcal F_3}=\boldsymbol{O}$, where $\textbf{D}|_{\mathcal F_3}$ denotes the restriction of $\textbf{D}$ in ${\mathcal F_3}$ and $\boldsymbol{O}$ is an $m_3 \times n_3$ zero matrix.
\end{theorem}

The following theorem is a variation of Theorem 10 in  \cite{egrw}.

\begin{theorem}\label{thm:combine with same distance}
Let $\mathcal F_{12}=[\gamma_{0},\gamma_{1},\ldots,\gamma_{n_1-1},\gamma_{n_1},\ldots,\gamma_{n_1+n_2-1}]$ be a $\gamma_{n_1+n_2-1}\times (n_1+n_2)$ Ferrers diagram, which induces a $\gamma_{n_1-1}\times n_1$ Ferrers diagram $\mathcal F_1=[\gamma_{0},\gamma_{1},\ldots,\gamma_{n_1-1}]$ and a $\gamma_{n_1+n_2-1}\times n_2$ Ferrers diagram $\mathcal F_2=[\gamma_{n_1},\gamma_{n_1+1},\ldots,\gamma_{n_1+n_2-1}]$ . Let $\mathcal F_3$ be a $\gamma_{n_1-1} \times \gamma_{n_1+n_2-1}$ full Ferrers diagram with $\gamma_{n_1-1} \gamma_{n_1+n_2-1}$ dots.
If there exists an $[\mathcal F_{12}, k, \delta]_q$ code, then there exists an $[\mathcal F, k, \delta]_q$ code $\cal D$, where
\begin{center}
$\mathcal F=\left(
  \begin{array}{cc}
    \mathcal F_1 & \mathcal F_3 \\
      & \mathcal F^t_2 \\
  \end{array}
\right)$,
\end{center}
satisfying that for any codeword $\textbf{D}\in {\cal D}$, $\textbf{D}|_{\mathcal F_3}=\boldsymbol{O}$, where $\textbf{D}|_{\mathcal F_3}$ denotes the restriction of $\textbf{D}$ in ${\mathcal F_3}$ and $\boldsymbol{O}$ is a $\gamma_{n_1-1} \times \gamma_{n_1+n_2-1}$ zero matrix.
\end{theorem}

\begin{proof} 
Let $\mathcal D_{12}$ be the given $[\mathcal F_{12}, k, \delta]_q$ code. Set
\begin{center}
$\mathcal D=\left\{\left(
   \begin{array}{cc}
     \boldsymbol{D}|_{\mathcal F_1} & \boldsymbol O \\
     \boldsymbol O & \boldsymbol{D}|_{\mathcal F^t_2} \\
   \end{array}
 \right):\boldsymbol{D}\in \mathcal D_{12}\right\},$
\end{center}
where $\boldsymbol{D}|_{\mathcal F_1}$ denotes the restriction of $\boldsymbol{D}$ in $\mathcal F_1$ and $\boldsymbol{D}|_{\mathcal F^t_2}$ denotes the restriction of $\boldsymbol{D}$ in $\mathcal F^t_2$.
Then $\mathcal D$ is an $[\mathcal F, k, \delta]_q$ code. One can easily verify the linearity and the dimension of $\mathcal D$. For any nonzero $\boldsymbol{D}\in \mathcal D_{12}$, we have
 \begin{align*}
 {\rm rank} \left(
   \begin{array}{cc}
     \boldsymbol{D}|_{\mathcal F_1} & \boldsymbol O \\
     \boldsymbol O & \boldsymbol{D}|_{\mathcal F^t_2} \\
   \end{array}
 \right)&={\rm rank}(\boldsymbol{D}|_{\mathcal F_1})+{\rm rank}(\boldsymbol{D}|_{\mathcal F_2})\\&\geq {\rm rank}(\boldsymbol{D}|_{\mathcal F_1}\mid \boldsymbol{D}|_{\mathcal F_2})={\rm rank}(\boldsymbol{D})\geq \delta,
 \end{align*}
 which proves the minimum rank distance of $\cal D$. \qed
\end{proof}

\subsubsection{Multilevel construction}

A matrix is said to be {\em in row echelon form} if each nonzero row has more leading zeros than the previous row. A matrix is {\em in reduced row echelon form} if (1) the leading coefficient of a row is always to the right of the leading coefficient of the previous row; (2) all leading coefficients are ones; (3) every leading coefficient is the only nonzero entry in its column.

A $k$-dimensional subspace $\mathcal U$ of $\mathbb F^n_q$ can be represented by a {\em $k\times n$ generator matrix} whose rows form a basis of $\mathcal U$. We usually
represent a codeword of a constant dimension code by such a matrix. There is exactly one such matrix in reduced row echelon form and it will be denoted by $E(\mathcal U)$ \cite{es09}.

The {\em identifying vector} $\boldsymbol v(\mathcal U)$ of a subspace $\mathcal U\in {\cal G}_q(n,k)$ is the binary vector of length $n$ and weight $k$ such that the $1's$ of $\boldsymbol v(\mathcal U)$ are in the positions (columns) where $E(\mathcal U)$ has its leading ones (of the rows). The {\em echelon Ferrers form} of a vector of length $n$ and weight $k$, $EF(\boldsymbol v)$, is the matrix in reduced row echelon form with leading entries (of rows) in the columns indexed by the nonzero entries of $\boldsymbol v$ and $\bullet$ (called a {\em dot}) in all entries which do not have terminals zeros or ones. 
\begin{example}
Consider the subspace $\mathcal U\in {\cal G}_2(7,3)$ with reduced row echelon form
\begin{center}
$E(\mathcal U)=\left(
     \begin{array}{ccccccc}
       1 & 0 & 0 & 0 & 1 & 1 & 0 \\
       0 & 0 & 1 & 0 & 1 & 0 & 1 \\
       0 & 0 & 0 & 1 & 0 & 1 & 1 \\
     \end{array}
   \right)$.
\end{center}
Its identifying vector $\boldsymbol{v}(\mathcal U)$ is $1011000$. For the identifying vector $\boldsymbol{v}(\mathcal U)=1011000$, its echelon Ferrers form $EF(\boldsymbol{v}(\mathcal U))$ is the following $3\times 7$ matrix:
\begin{center}
$EF(\boldsymbol{v}(\mathcal U))=\left(
  \begin{array}{ccccccc}
    1 & \bullet & 0 & 0 & \bullet & \bullet & \bullet \\
    0 & 0 & 1 & 0 & \bullet & \bullet & \bullet \\
    0 & 0 & 0 & 1 & \bullet & \bullet & \bullet \\
  \end{array}
\right)$.
\end{center}
\end{example}

To present the multilevel construction, the following two lemmas are crucial.

\begin{lemma}{\rm \cite[Lemma 2]{es09}}\label{lem:efc-2}
Let $\mathcal U,\mathcal V \in \mathcal G_q(n,k)$, $\mathcal U=rowspace(\boldsymbol U)$ and $\mathcal V=rowspace(\boldsymbol V)$, where $\boldsymbol U, \boldsymbol V\in \mathbb F_q^{k\times n}$ are in reduced row echelon forms. Then 
\begin{center}
$d_S(\mathcal U, \mathcal V) \geq d_H(\boldsymbol{v(U)}, \boldsymbol{v(V)})$.
\end{center}
\end{lemma}

\begin{lemma}{\rm(\cite{es09}, \cite[Proposition 1.2]{tr})}\label{lem:efc-1}
Let $\mathcal U, \mathcal{V}\in \mathcal G_q(n,k)$, $\mathcal U=rowspace(\boldsymbol U)$ and $\mathcal V=rowspace(\boldsymbol V)$, where $\boldsymbol U, \boldsymbol V\in \mathbb F_q^{k\times n}$ are in reduced row echelon forms.
If $\boldsymbol{v(U)} = \boldsymbol{v(V)}$, then
\begin{center}
$d_S(\mathcal U,\mathcal V)=2d_R(\boldsymbol C_{\boldsymbol U}, \boldsymbol C_{\boldsymbol V})$,
\end{center}
where $\boldsymbol C_{\boldsymbol U}$ and $\boldsymbol C_{\boldsymbol V}$ denote the submatrices of $\boldsymbol{U}$ and $\boldsymbol{V}$, respectively, without the columns of their pivots.
\end{lemma}

\begin{construction}[Multilevel construction \cite{es09}]\label{con:multi-level}
Let $\cal C$ be a binary Hamming code of length $n$, weight $k$ and minimum Hamming distance $2\delta$. For each codeword ${\boldsymbol C}\in {\cal C}$, its echelon Ferrers form is $EF({\boldsymbol C})$. All dots in $EF({\boldsymbol C})$ produce a Ferrers diagram ${\cal F}_{\boldsymbol C}$. If there exists an $[{\cal F}_{\boldsymbol C}, k_{\boldsymbol C}, \delta]_q$ code ${\cal D}_{\boldsymbol C}$ for each ${\boldsymbol C}\in {\cal C}$, then by Lemmas $\ref{lem:efc-2}$ and $\ref{lem:efc-1}$, the row spaces of the matrices in $\bigcup_{{\boldsymbol C}\in {\cal C}} {\cal D}_{\boldsymbol C}$ form a $(n,2\delta,k)_q$-CDC.
\end{construction}

\subsection{Combination of parallel construction and multilevel construction}

Although the multilevel construction is effective to construct CDCs, it is not known so far how to pick up identifying vectors such that the resulting CDCs are optimal. Also, only a few infinite families on optimal FDRMCs are known in the literature.
The following constructions help to reduce the choice of identifying vectors.

\subsubsection{The first construction}

\begin{construction}\label{comb}
Let $n\geq 2k$. Suppose that there exists an $(n,M_1,2\delta,k)_q$-CDC which is constructed via the multilevel construction satisfying that for any of its identifying vectors $\boldsymbol v=(\underbrace{\boldsymbol v^{(1)}}_{n-k}\mid \underbrace{\boldsymbol v^{(2)}}_{k})$, it holds that $wt(\boldsymbol v^{(1)})\geq s\geq \delta$, where $wt(\boldsymbol v^{(1)})$ is the weight of $\boldsymbol v^{(1)}$. If there exists a $(k\times (n-k),M_2,\delta,[0,s-\delta])_q$-GRMC, then
there exists an $(n,M_1+M_2,2\delta,k)_q$-CDC.
\end{construction}

\begin{proof}
Let $\mathcal C_1$ be the given $(n,M_1,2\delta,k)_q$-CDC and  $\mathcal D$ be a $(k\times (n-k),M_2,\delta,[0,s-\delta])_q$-GRMC. Set $\mathcal C_2=\{{\rm rowspace}(\boldsymbol B\mid \boldsymbol I_k): \boldsymbol B\in \mathcal D\}$. Let $\mathcal C=\mathcal C_1 \cup \mathcal C_2$. Then $\mathcal C$ is an $(n,M_1+M_2,2\delta,k)_q$-CDC.

It suffices to examine the subspace distance of $\mathcal C$. For any $\mathcal U \in \mathcal C_1$, let $(\underbrace{\boldsymbol A}_{n-k}|\underbrace{\boldsymbol D}_{k})$ be the reduced row echelon form of $\mathcal U$. We have ${\rm rank}(\boldsymbol A)\geq s$ by the fact that any identifying vector $\boldsymbol v$ satisfies wt$(\boldsymbol v^{(1)})\geq s$. For any $\mathcal V \in \mathcal C_2$,
\begin{align*}
d_S(\mathcal U,\mathcal V)&=2\cdot{\rm rank}\left(
                                                  \begin{array}{cc}
                                                    \boldsymbol A & \boldsymbol D \\
                                                    \boldsymbol B & \boldsymbol I_k \\
                                                  \end{array}
                                                \right)-2k\\
                                                &=2\cdot{\rm rank}\left(
                                                  \begin{array}{cc}
                                                    \boldsymbol A-\boldsymbol {DB} & \boldsymbol O \\
                                                    \boldsymbol B & \boldsymbol I_k \\
                                                  \end{array}
                                                \right)-2k\\
                                                &=2\cdot{\rm rank}(\boldsymbol A-\boldsymbol {DB})\\
                                                &\geq 2s-2\cdot{\rm rank}(\boldsymbol {DB})\\
                                                &\geq 2s-2\cdot{\rm rank}(\boldsymbol {B})\geq2s-2(s-\delta)=2\delta.
\end{align*}
For any $\mathcal U,\mathcal V \in \mathcal C_2$ and $\mathcal U\neq\mathcal V $, we know $d_S(\mathcal U,\mathcal V)\geq 2\delta$ by Lemma \ref{lem:efc-1}.\qed
\end{proof}

Construction \ref{comb} shows a criteria to combine the parallel construction and the multilevel construction to produce CDCs with large size. When $n\geq2k+\delta$ and $k\geq 2\delta$, applying Construction \ref{comb}, in what follows we shall present better lower bounds on ${\bar A}_q(n,2\delta,k)$ than that in Theorem \ref{cor:parallel}.

\begin{lemma}\label{mul-1}
Let $n\geq2k+\delta$ and $k\geq 2\delta$. Let $q$ be any prime power and
$$M_1=q^{(n-k)(k-\delta+1)}\frac{1-q^{-\lfloor\frac{k}{\delta}\rfloor\delta^2}}{1-q^{-\delta^2}}+
q^{(n-k-\delta)(k-\delta+1)}.$$
Then there exists an $(n,M_1,2\delta,k)_q$-CDC constructed via the multilevel construction satisfying that for any of its identifying vectors $\boldsymbol v=(\underbrace{\boldsymbol v^{(1)}}_{n-k}|\underbrace{\boldsymbol v^{(2)}}_{k})$, it holds that wt$(\boldsymbol v^{(1)})\geq k$, and this CDC contains a lifted MRD code $(n,q^{(n-k)(k-\delta+1)},2\delta,k)_q$-CDC as a subset.
\end{lemma}

\begin{proof}
We construct the set of identifying vectors of length $n$ as follows:
 \begin{center}
 $\mathcal A=\{(\underbrace{1\ldots1}_{k}\underbrace{0\ldots0}_{n-k})\}\cup\{(\boldsymbol u\mid \underbrace{1\ldots1}_{\delta}\underbrace{0\ldots0}_{n-k-\delta}):~\boldsymbol u\in \mathcal B\}$,
 \end{center}
where
\begin{center}
$\mathcal B=\{(\underbrace{1\ldots1}_{k-\delta}\underbrace{0\ldots0}_{\delta}), (\underbrace{1\ldots1}_{k-2\delta}\underbrace{0\ldots0}_{\delta}\underbrace{1\ldots1}_{\delta}),\ldots, (\underbrace{1\ldots1}_{k-\lfloor\frac{k}{\delta}\rfloor\delta}\underbrace{0\ldots0}_{\delta}
\underbrace{1\ldots1}_{(\lfloor\frac{k}{\delta}\rfloor-1)\delta})\}$.
\end{center}
The size of $\cal B$ is $\lfloor\frac{k}{\delta}\rfloor$, and each vector in $\cal B$ has size $k$ and weight $k-\delta$ (note that the $\delta$ zeros are always shifted to the left by $\delta$ positions). It is readily checked that for any $\boldsymbol v,\boldsymbol v'\in \mathcal A$, $\boldsymbol v\neq\boldsymbol v'$, we have $d_H(\boldsymbol v,\boldsymbol v')\geq 2\delta$.


 For any identifying vector $\boldsymbol v_i=(\underbrace{\boldsymbol v^{(1)}_{i}}_{n-k}|\underbrace{\boldsymbol v^{(2)}_{i}}_{k})=(\underbrace{1\ldots1}_{k-i\delta}\underbrace{0\ldots0}_{\delta}\underbrace{1\ldots1}_{i\delta}\underbrace{0\ldots0}_{n-k-\delta})\in \mathcal A$, $0\leq i\leq \lfloor\frac{k}{\delta}\rfloor$, we have wt$(\boldsymbol v^{(1)}_{i})\geq k$ because of $n-k\geq k+\delta$. The echelon Ferrers form of $\boldsymbol v_i$ is
 \begin{center}
$EF(\boldsymbol v_i)=\left(
                     \begin{array}{cccc}
                       \boldsymbol I_{k-i\delta} & \mathcal F_1 & \boldsymbol O_1 & \mathcal F_2 \\
                       \boldsymbol O_2 & \boldsymbol O_3 & \boldsymbol I_{i\delta} & \mathcal F_3 \\
                     \end{array}
                   \right)$,
\end{center}
where $\mathcal F_1$ is a $(k-i\delta)\times \delta$ full Ferrers diagram, $\mathcal F_2$ is a $(k-i\delta)\times (n-k-\delta)$ full Ferrers diagram, $\mathcal F_3$ is an $i\delta\times (n-k-\delta)$ full Ferrers diagram, $\boldsymbol O_1$ is a $(k-i\delta)\times i\delta$ zero matrix, $\boldsymbol O_2$ is an $i\delta\times (k-i\delta)$ zero matrix and $\boldsymbol O_3$ is an $i\delta\times \delta$ zero matrix.
For $0\leq i\leq \lfloor\frac{k}{\delta}\rfloor$, let \begin{center}
$\mathcal H_i=[\underbrace{k-i\delta,\ldots,k-i\delta}_{\delta},\underbrace{k,\ldots,k}_{n-k-\delta}]$
\end{center}
be a $k\times (n-k)$ Ferrers diagram, and
\begin{center}
 $\mathcal C_i=\left\{{\rm rowspace}\left(
                     \begin{array}{cccc}
                       \boldsymbol I_{k-i\delta} & \boldsymbol D_1 & \boldsymbol O_1 & \boldsymbol D_2 \\
                       \boldsymbol O_2 & \boldsymbol O_3 & \boldsymbol I_{i\delta} & \boldsymbol D_3 \\
                     \end{array}
                   \right): \left(
                               \begin{array}{cc}
                                 \boldsymbol D_1 & \boldsymbol D_2 \\
                                 \boldsymbol O & \boldsymbol D_3 \\
                               \end{array}
                             \right)
                   \in \mathcal C_{\mathcal H_i} \right\}$,
\end{center}
where $\mathcal C_{\mathcal H_i}$ is an $[\mathcal H_i,(n-k)(k-\delta+1)-i\delta^2,\delta]_q$ code when $0\leq i\leq \lfloor\frac{k}{\delta}\rfloor-1$, and $\mathcal C_{\mathcal H_i}$ is an $[\mathcal H_i,(n-k-\delta)(k-\delta+1),\delta]_q$ code when $i=\lfloor\frac{k}{\delta}\rfloor$.

Here we need to examine the existence of $\mathcal C_{\mathcal H_i}$. Consider the transpose of $\mathcal H_i$:
\begin{center}
$\mathcal H_i^t=[\underbrace{n-k-\delta,\ldots,n-k-\delta}_{i\delta},\underbrace{n-k,\ldots,n-k}_{k-i\delta}]$.
\end{center}
For the $(n-k)\times k$ Ferrers diagram $\mathcal H_i^t$, every column of $\mathcal H_i^t$ contains at least $n-k-\delta$ dots. Since $n-k-\delta\geq k$, by Theorem \ref{thm:shortening}, there exists an optimal $[\mathcal H_i^t,\delta]_q$ code. Now it remains to analyse its dimension. If $k-i\delta\geq\delta$, i.e, $i\leq\lfloor\frac{k}{\delta}\rfloor-1$, then by Theorem \ref{thm:shortening}, the optimal dimension is equal to the number of dots in $\mathcal H_i^t$ which are not contained in the rightmost $\delta-1$ columns: $i\delta(n-k-\delta)+(k-i\delta-\delta+1)(n-k)=(n-k)(k-\delta+1)-i\delta^2$. Thus there exists an optimal $[\mathcal H_i^t,(n-k)(k-\delta+1)-i\delta^2,\delta]_q$ code. If $0\leq k-i\delta<\delta$, i.e, $i=\lfloor\frac{k}{\delta}\rfloor$, then the optimal dimension is $(n-k-\delta)(k-\delta+1)$. Thus there exists an optimal $[\mathcal H_i^t,(n-k-\delta)(k-\delta+1),\delta]_q$ code.

Let $\mathcal C=\bigcup_{i=0}^{\lfloor\frac{k}{\delta}\rfloor}\mathcal C_i$. Then
\begin{align*}
|\mathcal C|&=\sum_{i=0}^{\lfloor\frac{k}{\delta}\rfloor-1}q^{(n-k)(k-\delta+1)-i\delta^2}+q^{(n-k-\delta)(k-\delta+1)}\\&
=q^{(n-k)(k-\delta+1)}\frac{1-q^{-\lfloor\frac{k}{\delta}\rfloor\delta^2}}{1-q^{-\delta^2}}+q^{(n-k-\delta)(k-\delta+1)}.
\end{align*}
For any $\mathcal U\in \mathcal C_i$ and $\mathcal V \in \mathcal C_j$, $0\leq i,j\leq \lfloor\frac{k}{\delta}\rfloor$, $i\neq j$, by Lemma \ref{lem:efc-2}, we have $d_s(\mathcal U,\mathcal V)\geq d_H(\boldsymbol v_i,\boldsymbol v_j)\geq 2\delta$. For any $\mathcal U, \mathcal V \in \mathcal C_i$, $0\leq i\leq \lfloor\frac{k}{\delta}\rfloor$, by Lemma \ref{lem:efc-1}, we have $d_s(\mathcal U,\mathcal V)\geq 2\delta$. Thus $\mathcal C$ is an $(n,M_1,2\delta,k)_q$-CDC. This CDC contains a lifted MRD code $(n,q^{(n-k)(k-\delta+1)}$, $2\delta,k)_q$-CDC as a subset which comes from the identifying vector $(\underbrace{1\ldots1}_{k}\underbrace{0\ldots0}_{n-k})$. \qed
\end{proof}

Applying Lemma \ref{mul-1} and Construction \ref{comb} with $s=k$, we obtain the following result.

\begin{theorem}\label{new-3}
Let $n\geq2k+\delta$ and $k\geq 2\delta$. Then
\begin{align*}
{\bar A}_q(n,2\delta,k)\geq q^{(n-k)(k-\delta+1)}\frac{1-q^{-\lfloor\frac{k}{\delta}\rfloor\delta^2}}{1-q^{-\delta^2}}&+q^{(n-k-\delta)(k-\delta+1)}\\&+A^R_q(k\times (n-k), \delta, [0,k-\delta]).
\end{align*}
\end{theorem}

Theorem \ref{new-3} together with the use of Proposition \ref{grmc} provides many new constant dimension codes with larger size than the previously best known codes in \cite{hkkw}. We list some of them in Table 1 in Appendix B.

When $\delta=2$, Theorem \ref{new-3} can be improved by using the following CDCs constructed in \cite{st15}.

\begin{lemma} {\rm \cite[Construction B, Corollary 27]{st15}} \label{stc4}
Let $n\geq2k+2$ and $k\geq 4$. Let $q$ be any prime power and
\begin{align*}
M_1= \sum_{j=1}^{\lfloor\frac{n-2}{k}\rfloor-1}\left(q^{(k-1)(n-jk)}+\frac{(q^{2(k-2)}-1)(q^{2(n-jk-1)}-1)}{(q^4-1)^2} q^{(k-3)(n-jk-2)+4}\right).
\end{align*}
Then there exists an $(n,M_1,4,k)_q$-CDC constructed via the multilevel construction satisfying that for any of its identifying vectors $\boldsymbol v=(\underbrace{\boldsymbol v^{(1)}}_{n-k}\mid \underbrace{\boldsymbol v^{(2)}}_{k}),$ it holds that wt$(\boldsymbol v^{(1)})\geq k-2$, and this CDC contains a lifted MRD code $(n,q^{(n-k)(k-\delta+1)},4,k)_q$-CDC as a subset.
\end{lemma}

\begin{proof} (sketch only)
We here employ the same set of identifying vectors as those in Construction B of \cite{st15} (note that the identifying vectors are exhibited in \cite{st15} recursively, and here we list all of them explicitly):
\begin{align*}
&{\cal A}=\{(\underbrace{0\ldots0}_{ik}\underbrace{1\ldots1}_{k}\underbrace{0\ldots0}_{n-(i+1)k}):0\leq i\leq \lfloor\frac{n-2}{k}\rfloor-2\}\ \cup \\
&\{(\underbrace{0\ldots0}_{ik}|\underbrace{\boldsymbol w}_{k}|\underbrace{\boldsymbol z}_{n-(i+1)k}):~\boldsymbol w\in \mathcal B,~\boldsymbol z\in \mathcal D_{i+1},~ 0\leq i\leq \lfloor\frac{n-2}{k}\rfloor-2\},
\end{align*}
where
\begin{equation}
\mathcal B=\{(\underbrace{1\ldots1}_{k-2}00), (\underbrace{1\ldots1}_{k-4}0011),\ldots,\boldsymbol u\},~~~ \label{}\nonumber
\boldsymbol u=\left \{
\begin {aligned}
&(00\underbrace{11\ldots1}_{k-2}),~~{\rm if}~k~{\rm is ~ even};\\
&(100\underbrace{1\ldots1}_{k-3}),~~{\rm if}~k~{\rm is ~ odd},\\
\end {aligned}
\right.
\end{equation}
and
\begin{equation*}
\mathcal D_{i+1}=\{(11\underbrace{00\ldots00}_{n-(i+1)k-2}), (0011\underbrace{00\ldots00}_{n-(i+1)k-4}),\ldots,\boldsymbol u'\},
\end{equation*}
\begin{equation}\nonumber
\boldsymbol u'=\left \{
\begin {aligned}
&(\underbrace{00\ldots00}_{n-(i+1)k-2}11),~~{\rm if}~n-(i+1)k~{\rm is ~ even};\\
&(\underbrace{00\ldots00}_{n-(i+1)k-3}110),~~{\rm if}~n-(i+1)k~{\rm is ~ odd}.\\
\end {aligned}
\right.
\end{equation}
The size of $\cal B$ is $\lfloor\frac{k}{2}\rfloor$, and each vector in $\cal B$ has size $k$ and weight $k-2$ (note that the two zeros are always shifted to the left by two positions). The size of ${\cal D}_{i+1}$ is $\lfloor\frac{n-(i+1)k}{2}\rfloor$, and each vector in ${\cal D}_{i+1}$ has size $n-(i+1)k$ and weight $2$ (note that the two ones are always shifted to the left by two positions).

Since $i\leq \lfloor\frac{n-2}{k}\rfloor-2$, we have $n-k\geq (i+1)k$. Then for any identifying vector $\boldsymbol v=(\underbrace{\boldsymbol v^{(1)}}_{n-k}|\underbrace{\boldsymbol v^{(2)}}_{k})$ shown above, one can check that wt$(\boldsymbol v^{(1)})\geq k-2$. By Construction B and Corollary 27 in \cite{st15}, these identifying vectors generate an $(n,M_1,4,k)_q$-CDC. This CDC contains a lifted MRD code $(n,q^{(n-k)(k-\delta+1)},4,k)_q$-CDC as a subset which comes from the identifying vector $(\underbrace{1\ldots1}_{k}\underbrace{0\ldots0}_{n-k})$. \qed
\end{proof}



\begin{theorem}\label{con4}
Let $n\geq2k+2$ and $k\geq 4$. Then
\begin{align*}
{\bar A}_q(n,4,k)\geq & \sum_{j=1}^{\lfloor\frac{n-2}{k}\rfloor-1}\left(q^{(k-1)(n-jk)}+\frac{(q^{2(k-2)}-1)(q^{2(n-jk-1)}-1)}{(q^4-1)^2} q^{(k-3)(n-jk-2)+4}\right) \\
& +A^R_q(k\times (n-k), 2, [0,k-4]).
\end{align*}
\end{theorem}
\begin{proof} Applying Lemma \ref{stc4} and Construction \ref{comb} with $s=k-2$. \qed
\end{proof}

With the aid of a computer, we compared the values of ${\bar A}_q(n,4,k)$ from Theorems \ref{new-3} and \ref{con4}. It seems that when $k>4$ and $n$ is large enough, Theorem \ref{con4} always produces better lower bound on ${\bar A}_q(n,4,k)$ than that in Theorem \ref{new-3} (see Table 2 in Appendix B for example).

More specially, for $\delta=4$ and $k=5$, Silberstein and Trautmann in \cite{st15} presented an $(n,M_1,4,5)_q$-CDC with larger size than that in Lemma \ref{stc4} by choosing identifying vectors more carefully.

\begin{lemma}{\rm \cite[Construction C-5, Theorem 35]{st15}}\label{cdc5}
Let $n\geq 12$ and $q$ be any prime power. If $x$ is even, then let
\begin{align*}
M(x)&:=q^{4(x-5)}+(q^{2x-10}+q^{2x-14})(q^{2x-14}+\frac{x-8}{2}q^{x-9})+(q^{2x-11}+q^{2x-13})\\& \times(\frac{x-8}{2}q^{x-10}+q^{2x-15})+(q^{2x-12}+q^{2x-13})(2q^{2x-16}+\frac{x-10}{2}q^{x-11})\\
&+(q^{2x-12}+q^{2x-14})\times\left[\sum_{i=3}^{\min\{\lceil\frac{q}{2}\rceil+2,\lfloor\frac{x-5}{2}\rfloor\}}
(iq^{2x-2i-12}+(\frac{x-6}{2}-i)q^{x-2i-7})\right.\\&\left.+\sum_{i=2}^{\min\{\lfloor\frac{q}{2}\rfloor+2,
\lceil\frac{x-7}{2}\rceil\}}
(iq^{2x-2i-13}+(\frac{x-6}{2}-i)q^{x-2i-8})\right].
\end{align*}
If $x$ is odd, then let
\begin{align*}
M(x)&:= q^{4(x-5)}+(q^{2x-10}+q^{2x-14})(q^{2x-14}+\frac{x-9}{2}q^{x-8}+q^{\frac{x-9}{2}})+(q^{2x-11}+q^{2x-13})\\& \times(\frac{x-9}{2}q^{x-9}+q^{2x-15}+q^{x-8})+(q^{2x-12}+q^{2x-13})
(q^{2x-16}+\frac{x-11}{2}q^{x-10}+q^{\frac{x-11}{2}})\\
&+(q^{2x-12}+q^{2x-14})\times\left[\sum_{i=3}^{\min\{\lceil\frac{q}{2}\rceil+2,\lfloor\frac{x-5}{2}\rfloor\}}
(iq^{2x-2i-12}+(\frac{x-7}{2}-i)q^{x-2i-6}+q^{\frac{x-7}{2}-i})\right.\\
&\left.+\sum_{i=2}^{\min\{\lfloor\frac{q}{2}\rfloor+2,\lceil\frac{x-7}{2}\rceil\}}
(iq^{2x-2i-13}+(\frac{x-7}{2}-i)q^{x-2i-7}+q^{x-i-7})\right].
\end{align*}
Let
$$M_1=\sum_{j=0}^{\lfloor\frac{n-12}{5}\rfloor} M(n-5j).$$
Then there exists an $(n,M_1,4,5)_q$-CDC constructed via the multilevel construction satisfying that for any of its identifying vectors $\boldsymbol v=(\underbrace{\boldsymbol v^{(1)}}_{n-5}\mid \underbrace{\boldsymbol v^{(2)}}_{5})$, it holds that wt$(\boldsymbol v^{(1)})\geq 3$, and this CDC contains a lifted MRD code $(n,q^{4(n-5)},4,5)_q$-CDC as a subset.
\end{lemma}

\begin{proof}(sketch only)
We here employ the same set of identifying vectors as those in Construction C-5 of \cite{st15} (note that the identifying vectors are exhibited in \cite{st15} recursively, and here we list all of them explicitly):
\begin{align*}
&\{(11111\mid \underbrace{0\ldots0)}_{n-5}\}\ \cup\\
&\{(\underbrace{0\ldots0}_{5j}\mid 11100\mid \boldsymbol u),(\underbrace{0\ldots0}_{5j}\mid 10011\mid \boldsymbol u):\ \boldsymbol u\in P_{j,\lceil\frac{n-5j-5}{2}\rceil+1},\ 0\leq j\leq \lfloor\frac{n-12}{5}\rfloor\}\ \cup\\
&\{(\underbrace{0\ldots0}_{5j}\mid11010\mid\boldsymbol u),(\underbrace{0\ldots0}_{5j}\mid01101\mid\boldsymbol u):\ \boldsymbol u\in P_{j,2},\ 0\leq j\leq \lfloor\frac{n-12}{5}\rfloor\}\ \cup\\
&\{(\underbrace{0\ldots0}_{5j}\mid 01110\mid \boldsymbol u),(\underbrace{0\ldots0}_{5j}\mid 10101\mid \boldsymbol u):\ \boldsymbol u\in P_{j,\lceil\frac{n-5j-5}{2}\rceil+2},\ 0\leq j\leq \lfloor\frac{n-12}{5}\rfloor\}\ \cup\\
&\{(\underbrace{0\ldots0}_{5j}\mid 00111\mid\boldsymbol u),(\underbrace{0\ldots0}_{5j}\mid11001\mid\boldsymbol u):\ \boldsymbol u\in P_{j,3},\ 0\leq j\leq \lfloor\frac{n-12}{5}\rfloor\}\ \cup\\
&\{(\underbrace{0\ldots0}_{5j}\mid 10110\mid \boldsymbol u),(\underbrace{0\ldots0}_{5j}\mid 01011\mid \boldsymbol u):\ 0\leq j\leq \lfloor\frac{n-12}{5}\rfloor,\\
&\boldsymbol u\in \left(\bigcup_{i=3}^{\min\{\lceil\frac{q}{2}\rceil+2,\lfloor\frac{n-5j-5}{2}\rfloor\}}P_{j,\lceil\frac{n-5j-5}{2}
\rceil+i}\right)\cup\left(\bigcup_{i=4}^{\min\{\lfloor\frac{q}{2}\rfloor+3,\lceil\frac{n-5j-5}{2}\rceil\}}P_{j,i}
\right)\},
\end{align*}
where each vector in $P_{j,l}$ has size $n-5j-5$ and weight $2$; if $n-5j-5$ is even, then the positions of ones in vectors from $P_{j,2},\ldots,P_{j,n-5j-5}$ correspond to a one-factorization of the complete graph $K_{n-5j-5}$; if $n-5j-5$ is odd, then the positions of ones in vectors from $P_{j,1},P_{j,2},\ldots,P_{j,n-5j-5}$ correspond to a near one-factorization of the complete graph $K_{n-5j-5}$ (see \cite[Construction C-5]{st15} for more details).

Since $j\leq\lfloor\frac{n-12}{5}\rfloor$, we have $5j+5< n-5$. Then for any identifying vector $\boldsymbol v=(\underbrace{\boldsymbol v^{(1)}}_{n-5}|\underbrace{\boldsymbol v^{(2)}}_{5})$ shown above, one can check that wt$(\boldsymbol v^{(1)})\geq 3$. By Construction C-5 and Theorem 35 in \cite{st15}, these identifying vectors generate an $(n,M_1,4,5)_q$-CDC. This CDC contains a lifted MRD code $(n,q^{4(n-5)},4,5)_q$-CDC as a subset which comes from the identifying vector $(11111\mid\underbrace{0\ldots0}_{n-5})$. \qed
\end{proof}

\begin{theorem} \label{cdc45}
Let $M_1$ be as in Lemma $\ref{cdc5}$ and $n\geq 12$. Then
\begin{center}
 ${\bar A}_q(n,4,5)\geq M_1+q^4+q^3+q^2+q+1$.
\end{center}
\end{theorem}
\begin{proof}
Applying Lemma \ref{cdc5} and Construction \ref{comb} with $s=3$ , we have
\begin{center}
${\bar A}_q(n,4,5)\geq M_1+A^R_q(5\times(n-5),2,\{0,1\})$.
\end{center}
By Proposition \ref{grmc} and Theorem \ref{distribution}, one can check that
\begin{center}
$A^R_q(5\times(n-5),2,\{0,1\})\geq q^4+q^3+q^2+q+1$.
\end{center}
This completes the proof. \qed
\end{proof}

Theorem \ref{cdc45} provides many new $(n,4,5)_q$-CDCs with larger size than the previously best known codes in \cite{hkkw}. We list some of them in Table 2 in Appendix B. Compared with Theorem \ref{cdc45}, Theorems \ref{new-3} and \ref{con4} produce worse lower bounds of ${\bar A}_q(n,4,5)$ for those $n$ in Table 2.

\subsubsection{The second construction}

In Construction \ref{comb}, we start from a multilevel construction, and then choose an appropriate parallel construction based on the identifying vectors in the multilevel construction. Actually, we can also start from a parallel construction, and then choose appropriate identifying vectors to use the multilevel construction.

\begin{construction}\label{new}
Let $n\geq2k$ and $k\geq 2\delta$. If there exists a $(k\times (n-k),M,\delta,[0,k-\delta])_q$-GRMC, then there exists an $(n, q^{(n-k)(k-\delta+1)}+M+q^{\max\{l_1,l_2\}}+q^{(n-k-\delta)(k-2\delta+1)}, 2\delta,k)_q$-CDC, where
 \begin{equation}
\label{}\nonumber
l_1=\left \{
\begin {aligned}
&(k-\delta)\delta+n-k-\delta,~~{\rm if}~n\geq k+3\delta;\\
&\delta(n-4\delta+2),~~{\rm if}~n<k+3\delta,\\
\end {aligned}
\right.
\end{equation}
and $$l_2= \max_{1\leq j\leq \delta-1} \{\min\{(\delta-j+1)(k-\delta),(j+1)(n-k-\delta)\}\}.$$
This CDC contains a lifted MRD code $(n,q^{(n-k)(k-\delta+1)},2\delta,k)_q$-CDC as a subset.
\end{construction}

\begin{proof}
Let $\mathcal D_1$ be an MRD$[k\times (n-k),\delta]_q $ code and $\mathcal D_2$ be a $(k\times (n-k),M,\delta,[0,k-\delta])_q $-GRMC. Set $$\mathcal C_1=\{{\rm rowspace}(\boldsymbol I_k\mid \boldsymbol A):\boldsymbol A\in \mathcal D_1\}$$ and $$\mathcal C_2=\{{\rm rowspace}(\boldsymbol B\mid \boldsymbol I_k): \boldsymbol B \in \mathcal D_2\}.$$
Since $n\geq 2k$, $|\mathcal C_1|=q^{(n-k)(k-\delta+1)}$. Note that the identifying vector of each codeword in $\mathcal C_1$ is $\boldsymbol n=(\underbrace{1\ldots1}_{k}\underbrace{0\ldots0}_{n-k})$.
Now we take two new identifying vectors $\boldsymbol n_1$ and $\boldsymbol n_2$.

\underline{Step 1.}
Take $$\boldsymbol n_1=(\underbrace{1\ldots1}_{k-\delta}\underbrace{0\ldots0}_{\delta}\underbrace{1\ldots1}_{\delta}
\underbrace{0\ldots0}_{n-k-\delta}).$$ Note that $d_H(\boldsymbol n,\boldsymbol n_1)=2\delta$. Then
\begin{center}
$EF(\boldsymbol n_1)=\left(
\begin{array}{cccc}
\boldsymbol I_{k-\delta} & \mathcal F_1 & \boldsymbol O_1 & \mathcal F_3 \\
\boldsymbol O_2 & \boldsymbol O_3 & \boldsymbol I_{\delta} & \mathcal F_2 \\
\end{array}
\right)$,
\end{center}
where $\mathcal F_1$ is a $(k-\delta)\times \delta$ full Ferrers diagram, $\mathcal F_2$ is a $\delta\times (n-k-\delta)$ full Ferrers diagram, $\mathcal F_3$ is a $(k-\delta)\times (n-k-\delta)$ full Ferrers diagram, $\boldsymbol O_1$ is a $(k-\delta)\times \delta$ zero matrix, $\boldsymbol O_2$ is a $\delta\times (k-\delta)$ zero matrix and $\boldsymbol O_3$ is a $\delta\times \delta$ zero matrix. Let
$$\mathcal F=\left(
\begin{array}{cc}
\mathcal F_1 & \mathcal F_3 \\
& \mathcal F_2 \\
\end{array}
\right)$$ be a $k\times (n-k)$ Ferrers diagram and
\begin{center}
 $\mathcal C_3=\left\{{\rm rowspace}\left(
                     \begin{array}{cccc}
                       \boldsymbol I_{k-\delta} & \boldsymbol D_1 & \boldsymbol O_1 & \boldsymbol D_3\\
                       \boldsymbol O_2 & \boldsymbol O_3 & \boldsymbol I_{\delta} & \boldsymbol D_2 \\
                     \end{array}
                   \right): \left(
                                                                \begin{array}{cc}
                                                                  \boldsymbol D_1 & \boldsymbol D_3 \\
                                                                  \boldsymbol O_3 & \boldsymbol D_2 \\
                                                                \end{array}
                                                              \right)
                   \in \mathcal D_{\mathcal F} \right\}$,
\end{center}
where $\mathcal D_{\mathcal F}$ is a Ferrers diagram rank-metric code in $\mathcal F$ with minimum rank distance $\delta$, which will be constructed as follows in two different ways.

(1) We claim that there exists an $[\mathcal F,(k-\delta)\delta+n-k-\delta,\delta]_q$ code $\mathcal D'_{\mathcal F}$ when $n\geq k+3\delta$, and there exists an $[\mathcal F,\delta(n-4\delta+2),\delta]_q$ code $\mathcal D'_{\mathcal F}$ when $2k\leq n< k+3\delta$.

To examine the existence of such FDRMC codes, take an $(n-k-\delta)\times 2\delta$ Ferrers diagram
\begin{center}
$\mathcal F_{12}=(\mathcal F_1\mid \mathcal F^t_2)=[\underbrace{k-\delta,\ldots,k-\delta}_{\delta},\underbrace{n-k-\delta,\ldots,n-k-\delta}_{\delta}]$.
\end{center}
Note that
\begin{center}
$\mathcal F^t_{12}=[\underbrace{\delta,\ldots,\delta}_{n-2k},\underbrace{2\delta,\ldots,2\delta}_{k-\delta}]$.
\end{center}
By Theorem \ref{thm:shortening}, when $n-k-\delta\geq 2\delta$, i.e., $n\geq k+3\delta$, there exists an optimal $[\mathcal F_{12},(k-\delta)\delta+n-k-\delta,\delta]_q$ code, and when $n< k+3\delta$, since $k-\delta\geq \delta$, there exists an optimal $[\mathcal F^t_{12},\delta(n-4\delta+2),\delta]_q$ code, which yields an $[\mathcal F_{12},\delta(n-4\delta+2),\delta]_q$ code. Now applying Theorem \ref{thm:combine with same distance}, we obtain an $[\mathcal F,(k-\delta)\delta+n-k-\delta,\delta]_q$ code $\mathcal D'_{\mathcal F}$ when $n\geq k+3\delta$, and an $[\mathcal F,\delta(n-4\delta+2),\delta]_q$ code $\mathcal D'_{\mathcal F}$ when $2k\leq n< k+3\delta$. Note that each codeword in $\mathcal D'_{\mathcal F}$ is of the form
$\left(
\begin{array}{cc}
* & \boldsymbol O_4 \\
\boldsymbol O_3 & * \\
\end{array}
\right),$ where $\boldsymbol O_4$ is a $(k-\delta)\times (n-k-\delta)$ zero matrix.

(2) We claim that there exists an
$$[\mathcal F,\max\limits_{1\leq j\leq \delta-1} \left\{\min\{(\delta-j+1)(k-\delta),(j+1)(n-k-\delta)\}\right\},\delta]_q$$ code $\mathcal D''_{\mathcal F}$. Its existence comes from Theorem \ref{thm:combine with same dim} by using an optimal $[\mathcal F_1, (\delta-j+1)(k-\delta),j]_q$ code, i.e., an MRD$[(k-\delta)\times \delta,j]_q$ code, and an optimal $[\mathcal F_2, (j+1)(n-k-\delta),\delta-j]_q$ code, i.e., an MRD$[\delta\times(n-k-\delta),\delta-j]_q$ code (note that $\delta\leq n-k-\delta$ since $n\geq2k$ and $k\geq 2\delta$). Each codeword in $\mathcal D''_{\mathcal F}$ is of the form
$\left(
\begin{array}{cc}
* & \boldsymbol O_4 \\
\boldsymbol O_3 & * \\
\end{array}
\right).$

Let $l_1$ and $l_2$ be as in the assumption. If $l_1\geq l_2$, then take $\mathcal D_{\cal F}=\mathcal D'_{\cal F}$. Otherwise, take $\mathcal D_{\cal F}=\mathcal D''_{\cal F}$. Thus $|\mathcal C_3|=|\mathcal D_{\cal F}|=q^{\max\{l_1,l_2\}}$. Note that $\boldsymbol D_3=\boldsymbol O_4$ in both cases.

\underline{Step 2.} Take $$\boldsymbol n_2=(\underbrace{1\ldots1}_{k-2\delta}\underbrace{0\ldots0}_{\delta}\underbrace{1\ldots1}_{\delta}
\underbrace{1\ldots1}_{\delta}\underbrace{0\ldots0}_{n-k-\delta}).$$
Note that $d_H(\boldsymbol n,\boldsymbol n_2)=d_H(\boldsymbol n_1,\boldsymbol n_2)=2\delta$. Then
\begin{center}
$EF(\boldsymbol n_2)=\left(
   \begin{array}{ccccc}
     \boldsymbol I_{k-2\delta} & \mathcal F_4 & \boldsymbol O_5 & \boldsymbol O_5 & \mathcal F_5 \\
     \boldsymbol O_6 & \boldsymbol O_7 & \boldsymbol I_{\delta} & \boldsymbol O_7 & \mathcal F_6 \\
     \boldsymbol O_6 & \boldsymbol O_7 & \boldsymbol O_7 & \boldsymbol I_{\delta} & \mathcal F_7 \\
   \end{array}
 \right)$,
\end{center}
where $\mathcal F_4$ is a $(k-2\delta)\times \delta$ full Ferrers diagram, $\mathcal F_5$ is a $(k-2\delta)\times (n-k-\delta)$ full Ferrers diagram, $\mathcal F_6$ and $\mathcal F_7$ are $\delta\times (n-k-\delta)$ full Ferrers diagrams, $\boldsymbol O_5$ is a $(k-2\delta)\times \delta$ zero matrix, $\boldsymbol O_6$ is a $\delta\times (k-2\delta)$ zero matrix and $\boldsymbol O_7$ is a $\delta\times \delta$ zero matrix. Let
\begin{center}
$\mathcal C_4=\left\{{\rm rowspace}\left(
   \begin{array}{ccccc}
     \boldsymbol I_{k-2\delta} & \boldsymbol O_5 & \boldsymbol O_5 & \boldsymbol O_5 & \boldsymbol D_4 \\
     \boldsymbol O_6 & \boldsymbol O_7 & \boldsymbol I_{\delta} & \boldsymbol O_7 & \boldsymbol O_8 \\
     \boldsymbol O_6 & \boldsymbol O_7 & \boldsymbol O_7 & \boldsymbol I_{\delta} & \boldsymbol D_5 \\
   \end{array}
 \right):\left(
           \begin{array}{c}
             \boldsymbol D_4 \\
             \boldsymbol D_5 \\
           \end{array}
         \right)\in \mathcal D_3
 \right\}$,
\end{center} where $\mathcal D_3$ is an MRD$[(k-\delta)\times (n-k-\delta),\delta]_q$ code. Since $n-k\geq k$, $|\mathcal C_4|=|\mathcal D_3|=q^{(n-k-\delta)(k-2\delta+1)}$.

Let $\mathcal C=\mathcal C_1\cup\mathcal C_2\cup\mathcal C_3\cup\mathcal C_4$. Then $\mathcal C$ is an $(n, q^{(n-k)(k-\delta+1)}+M+q^{\max\{l_1,l_2\}}+q^{(n-k-\delta)(k-2\delta+1)}, 2\delta,k)_q$-CDC.

It suffices to examine the subspace distance of $\cal C$. Let $\mathcal U_i\in \mathcal C_i$ for $1\leq i\leq 4$. Since $d_H(\boldsymbol n,\boldsymbol n_1)=d_H(\boldsymbol n,\boldsymbol n_2)=d_H(\boldsymbol n_1,\boldsymbol n_2)=2\delta$, by Lemma \ref{lem:efc-2}, $d_S(\mathcal U_1,\mathcal U_3)$, $d_S(\mathcal U_1,\mathcal U_4)$ and $d_S(\mathcal U_3,\mathcal U_4)$ are all no less than $2\delta$. Applying Construction \ref{comb} with $s=k$, we have $d_S(\mathcal U_1,\mathcal U_2)\geq 2\delta$. Since $k\geq2\delta$,
\begin{align*}
d_S(\mathcal U_2,\mathcal U_3)&
=2\cdot {\rm rank} \left(
                                        \begin{array}{c;{2pt/2pt}c}
                                          \begin{array}{ccc}
                                             \boldsymbol I_{k-\delta} & \boldsymbol D_1 & \boldsymbol O_{9} \\
                                             \boldsymbol O_2 & \boldsymbol O_3 & \boldsymbol D_6
                                           \end{array}
                                           & \begin{array}{c}
                                               \boldsymbol O_{10}  \\
                                               \boldsymbol D_7
                                             \end{array}
                                            \\ \hdashline[2pt/2pt]
                                         \boldsymbol B & \boldsymbol I_k \\
                                        \end{array}
                                      \right)-2k\\
& \geq 2 \left({\rm rank} \left(\begin{array}{ccc}
\boldsymbol I_{k-\delta} & \boldsymbol D_1 & \boldsymbol O_{9} \\
\end{array}\right)+{\rm rank} (\boldsymbol I_k)
\right)-2k=2(k-\delta)\geq 2\delta,
\end{align*}
where $(\underbrace{\boldsymbol O_{9}}_{n-2k}\mid \underbrace{\boldsymbol O_{10}}_{k})=(\boldsymbol O_{1}\mid \boldsymbol O_{4})$ and $(\underbrace{\boldsymbol D_6}_{n-2k}\mid \underbrace{\boldsymbol D_7}_{k})=(\boldsymbol I_{\delta}\mid \boldsymbol D_{2})$.
Similarly we have
\begin{center}
 $d_S(\mathcal U_2,\mathcal U_4)=2\cdot {\rm rank}
\left(
  \begin{array}{c;{2pt/2pt}c}
    \begin{array}{cccc}
     \boldsymbol I_{k-2\delta} & \boldsymbol O_5 & \boldsymbol O_5 & \boldsymbol D_8  \\
     \boldsymbol O_6 & \boldsymbol O_7 & \boxed{\boldsymbol I_{\delta}} & \boldsymbol O_{11}  \\
     \boldsymbol O_6 & \boldsymbol O_7 & \boldsymbol O_7 & \boldsymbol D_{10}
    \end{array}
     & \begin{array}{c}
                                                        \boldsymbol D_9 \\
                                                       \boldsymbol O_{12} \\
                                                       \boldsymbol D_{11}
       \end{array}
      \\ \hdashline[2pt/2pt]
 \boldsymbol B & \boxed{\boldsymbol I_k} \\
  \end{array}
\right)-2k\geq 2\delta$,
\end{center}
where $(\underbrace{\boldsymbol D_{8}}_{n-2k}\mid \underbrace{\boldsymbol D_{9}}_{k})=(\boldsymbol O_{5}\mid \boldsymbol D_{4})$, $(\underbrace{\boldsymbol O_{11}}_{n-2k}\mid\underbrace{\boldsymbol O_{12}}_{k})=(\boldsymbol O_{7}\mid \boldsymbol O_{8})$ and $(\underbrace{\boldsymbol D_{10}}_{n-2k}\mid\underbrace{\boldsymbol D_{11}}_{k})=(\boldsymbol I_{\delta}\mid\boldsymbol D_{5})$.
Finally, by Lemma \ref{lem:efc-1}, for any $\mathcal U,\mathcal V \in \mathcal C_i$ and $\mathcal U\neq \mathcal V$, $i\in\{1,2,3,4\}$, $d_S(\mathcal U,\mathcal V)\geq 2\delta$. \qed
\end{proof}

In the proof of Construction \ref{new}, three identifying vectors are used: $(\underbrace{1\ldots1}_{k}\underbrace{0\ldots0}_{n-k})$, $ (\underbrace{1\ldots1}_{k-\delta}\underbrace{0\ldots0}_{\delta}\underbrace{1\ldots1}_{\delta}
\underbrace{0\ldots0}_{n-k-\delta})$ and $(\underbrace{1\ldots1}_{k-2\delta}\underbrace{0\ldots0}_{\delta}\underbrace{1\ldots1}_{\delta}
\underbrace{1\ldots1}_{\delta}\underbrace{0\ldots0}_{n-k-\delta})$. Since $n\geq 2k$, for each identifying vectors $\boldsymbol v=(\underbrace{\boldsymbol v^{(1)}}_{n-k}\mid \underbrace{\boldsymbol v^{(2)}}_{k})$, it holds that $wt(\boldsymbol v^{(1)})\geq k-\delta$. Therefore, if we apply Construction \ref{comb} via the three identifying vectors to produce CDCs, then we have to construct a $(k\times (n-k),M_2,\delta,[0,k-2\delta])_q$-GRMC. However, in Construction \ref{new}, we can use a $(k\times (n-k),M,\delta,[0,k-\delta])_q$-GRMC. Generally $M\geq M_2$. From this point of view, Construction \ref{new} is better than Construction \ref{comb}. But in Construction \ref{comb}, one can choose other identifying vectors flexibly to change the details of the multilevel construction. From this point of view, Construction \ref{comb} is better. Anyway, compared with Construction \ref{comb}, Construction \ref{new} is easier to be used since its statement does not rely on the choice of identifying vectors.

Construction \ref{new} together with the use of Proposition \ref{grmc} provides many new constant dimension codes with larger size than the previously best known codes in \cite{hkkw}. We list some of them in Table 3 in Appendix B. Especially, when $n=2k=4\delta$, we obtain new lower bound on  ${\bar A}_q(n,2\delta,k)$.


\begin{corollary}\label{cor2}
${\bar A}_q(4\delta,2\delta,2\delta)\geq q^{2\delta(\delta+1)}+(q^{2\delta}-1)\begin{bmatrix}
 2\delta \\
 \delta \\
 \end{bmatrix}_q+q^{l}+q^{\delta}+1$, where
 \begin{equation}
\label{}\nonumber
l=\left \{
\begin {aligned}
&2,~~{\rm if}~\delta=1;\\
&(\lfloor\frac{\delta}{2}\rfloor+1)\delta,~~{\rm if}~\delta\geq 2.\\
\end {aligned}
\right.
\end{equation}
\end{corollary}

\proof By Theorems \ref{grmc} and \ref{distribution}, $A^R_q(2\delta\times 2\delta,\delta,[0,\delta]) \geq (q^{2\delta}-1)\begin{bmatrix}
 2\delta \\
 \delta \\
 \end{bmatrix}_q+1$. Apply Construction \ref{new} with $n=4\delta$ and $k=2\delta$. Then $l_1=2\delta$, $l_2=(\lfloor\frac{\delta}{2}\rfloor+1)\delta$, and so $\max\{l_1,l_2\}=l$. \qed


Combining Corollary \ref{cor2} and Theorem \ref{upper-h}(3) with $n=4\delta$ and $k=2\delta$, we have the following corollary.

\begin{corollary}\label{cor422} Let $\delta\geq 2$. Then
\begin{center}
$q^{2\delta(\delta+1)}+(q^{2\delta}-1){\small\begin{bmatrix}
 2\delta \\
 \delta \\
 \end{bmatrix}_q}+q^{(\lfloor\frac{\delta}{2}\rfloor+1)\delta}+q^{\delta}+1\leq \bar{A}_q(4\delta,~2\delta,~2\delta)\leq q^{2\delta(\delta+1)}+(q^{2\delta}+q^{\delta}){\small \begin{bmatrix}
 2\delta \\
 \delta \\
 \end{bmatrix}_q}+1$.
\end{center}
\end{corollary}

 \begin{remark}\label{4delta}
We can calculate the ratio between the lower bound and the upper bound of the CDCs in Corollary $\ref{cor422}$:
\begin{align*}
\frac{{\rm the~lower~ bound~ of~}\mathcal C}{{\rm the~ upper~ bound~ of~}\mathcal C}&=\frac{q^{2\delta(\delta+1)}+(q^{2\delta}-1){\small\begin{bmatrix}
 2\delta \\
 \delta \\
 \end{bmatrix}_q}+q^{(\lfloor\frac{\delta}{2}\rfloor+1)\delta}+q^{\delta}+1}{q^{2\delta(\delta+1)}+(q^{2\delta}+q^{\delta}){\small \begin{bmatrix}
 2\delta \\
 \delta \\
 \end{bmatrix}_q}+1}\\&=1-\frac{(q^{\delta}+1){\small\begin{bmatrix}
 2\delta \\
 \delta \\
 \end{bmatrix}_q}-q^{(\lfloor\frac{\delta}{2}\rfloor+1)\delta}-q^{\delta}}{q^{2\delta(\delta+1)}+(q^{2\delta}+q^{\delta}){\small \begin{bmatrix}
 2\delta \\
 \delta \\
 \end{bmatrix}_q}+1}\geq \frac{4642}{4797}> 0.967688.
\end{align*}
Furthermore, for $\delta\geq 3$,
\begin{align*}
\frac{{\rm the~lower~ bound~ of~}\mathcal C}{{\rm the~ upper~ bound~ of~}\mathcal C}\geq \frac{16865174}{16877657}> 0.99926.
\end{align*}
To ensure smooth reading of the paper, we move the proof to Appendix \ref{app-A}.
 \end{remark}

There is no systematic approach so far to give a $(4\delta,2\delta,2\delta)_q$-CDC attaining the lower bound in Corollary \ref{cor422} for general $\delta$. In principle, people can always pick up suitable identifying vectors and then use the multilevel construction to construct an optimal $(4\delta,2\delta,2\delta)_q$-CDC. However, how to choose identifying vectors effectively is still an open and different problem. The combination of the parallel construction and the multilevel construction helps to weaken the requirement for identifying vectors and provides good constant dimension codes with large size.

\section{Concluding remarks}

Constructions \ref{comb} and \ref{new} for CDCs are established in this paper by combining the parallel construction and the multilevel construction. How to choose identifying vectors compatible with a given parallel construction as many as possible is still an open problem. This paper initials the study.

GRMCs play a fundamental role in our constructions. It is meaningful to investigate various constructions for GRMCs as an independent research topic.

\appendix
\section{Appendix}\label{app-A}

\textbf{Proof of Remark \ref{4delta}}\ \ Write
\begin{center}
$f(\delta)=\frac{(q^{\delta}+1){\small\begin{bmatrix}
 2\delta \\
 \delta \\
 \end{bmatrix}_q}-q^{(\lfloor\frac{\delta}{2}\rfloor+1)\delta}-q^{\delta}}
 {q^{2\delta(\delta+1)}+(q^{2\delta}+q^{\delta}){\small \begin{bmatrix}
 2\delta \\
 \delta \\
 \end{bmatrix}_q}+1}$.
\end{center}
We claim that given any prime power $q$, $f(\delta)$ is a non-increasing function on $\delta$ for any $\delta\geq 2$. We have
\begin{align*}
f(\delta)&-f(\delta+1)
\\&=\frac{\left((q^{\delta}+1){\small\begin{bmatrix}
 2\delta \\
 \delta \\
 \end{bmatrix}_q}-q^{(\lfloor\frac{\delta}{2}\rfloor+1)\delta}-q^{\delta}\right)\left(q^{2(\delta+1)(\delta+2)}+(q^{2\delta+2}+q^{\delta+1}){\small \begin{bmatrix}
 2\delta+2 \\
 \delta+1 \\
 \end{bmatrix}_q}+1\right)}{\left(q^{2(\delta+1)(\delta+2)}+(q^{2\delta+2}+q^{\delta+1}){\small \begin{bmatrix}
 2\delta+2 \\
 \delta+1 \\
 \end{bmatrix}_q}+1\right)\left(q^{2\delta(\delta+1)}+(q^{2\delta}+q^{\delta}){\small \begin{bmatrix}
 2\delta \\
 \delta \\
 \end{bmatrix}_q}+1\right)}
\\&-\frac{\left((q^{\delta+1}+1){\small\begin{bmatrix}
 2\delta +2 \\
 \delta +1\\
 \end{bmatrix}_q}-q^{(\lfloor\frac{\delta+1}{2}\rfloor+1)(\delta+1)}-q^{\delta+1}\right)\left(q^{2\delta(\delta+1)}+(q^{2\delta}+q^{\delta}){\small \begin{bmatrix}
 2\delta \\
 \delta \\
 \end{bmatrix}_q}+1\right)}{\left(q^{2(\delta+1)(\delta+2)}+(q^{2\delta+2}+q^{\delta+1}){\small \begin{bmatrix}
 2\delta+2 \\
 \delta+1 \\
 \end{bmatrix}_q}+1\right)\left(q^{2\delta(\delta+1)}+(q^{2\delta}+q^{\delta}){\small \begin{bmatrix}
 2\delta \\
 \delta \\
 \end{bmatrix}_q}+1\right)}.
\end{align*}
Let
\begin{align*}
g(\delta)&={\left((q^{\delta}+1){\small\begin{bmatrix}
 2\delta \\
 \delta \\
 \end{bmatrix}_q}-q^{(\lfloor\frac{\delta}{2}\rfloor+1)\delta}-q^{\delta}\right)\left(q^{2(\delta+1)(\delta+2)}+(q^{2\delta+2}+q^{\delta+1}){\small \begin{bmatrix}
 2\delta+2 \\
 \delta+1 \\
 \end{bmatrix}_q}+1\right)}
\\&-{\left((q^{\delta+1}+1){\small\begin{bmatrix}
 2\delta +2 \\
 \delta +1\\
 \end{bmatrix}_q}-q^{(\lfloor\frac{\delta+1}{2}\rfloor+1)(\delta+1)}-q^{\delta+1}\right)\left(q^{2\delta(\delta+1)}+(q^{2\delta}+q^{\delta}){\small \begin{bmatrix}
 2\delta \\
 \delta \\
 \end{bmatrix}_q}+1\right)}.
\end{align*}
It suffices to verify $g(\delta)\geq 0$ for any $\delta\geq 2$.

Since $q^{\delta+j}-1\geq (q^{j}-1)q^{\delta}$, when $\delta\geq3$, we have ${\small\begin{bmatrix}
 2\delta \\
 \delta \\
 \end{bmatrix}_q}=\frac{(q^{2\delta}-1)\cdots (q^{\delta+1}-1)}{(q^{\delta}-1)\cdots(q-1)}\geq q^{\delta^2}\geq q^{(\lfloor\frac{\delta}{2}\rfloor+1)\delta}+q^{\delta}$. When $\delta=2$, ${\small\begin{bmatrix}
 4 \\
 2 \\
 \end{bmatrix}_q}=q^4+q^3+2q^2+q+1>q^{(\lfloor\frac{2}{2}\rfloor+1)2}+q^{2}$. So ${\small\begin{bmatrix}
 2\delta \\
 \delta \\
 \end{bmatrix}_q}\geq q^{(\lfloor\frac{\delta}{2}\rfloor+1)\delta}+q^{\delta}$ for any $\delta\geq2$. It follows that
\begin{align*}
g(\delta)&\geq q^{\delta}{\small\begin{bmatrix}
 2\delta \\
 \delta \\
 \end{bmatrix}_q}\left(q^{2(\delta+1)(\delta+2)}+(q^{2\delta+2}+q^{\delta+1}){\small \begin{bmatrix}
 2\delta+2 \\
 \delta+1 \\
 \end{bmatrix}_q}\right)-\left((q^{\delta+1}+1){\small\begin{bmatrix}
 2\delta +2 \\
 \delta +1\\
 \end{bmatrix}_q}\right)\times \\&\left(q^{2\delta(\delta+1)}+(q^{2\delta}+q^{\delta}){\small \begin{bmatrix}
 2\delta \\
 \delta \\
 \end{bmatrix}_q}+1\right).
\end{align*}
Write
\begin{align*}
g_1(\delta)&= q^{\delta}{\small\begin{bmatrix}
 2\delta \\
 \delta \\
 \end{bmatrix}_q}(q^{2\delta+2}+q^{\delta+1}){\small \begin{bmatrix}
 2\delta+2 \\
 \delta+1 \\
 \end{bmatrix}_q}-(q^{\delta+1}+1){\small\begin{bmatrix}
 2\delta +2 \\
 \delta +1\\
 \end{bmatrix}_q}(q^{2\delta}+q^{\delta}){\small \begin{bmatrix}
 2\delta \\
 \delta \\
 \end{bmatrix}_q}
\end{align*}
and
\begin{align*}
g_2(\delta)&= q^{\delta}{\small\begin{bmatrix}
 2\delta \\
 \delta \\
 \end{bmatrix}_q}q^{2(\delta+1)(\delta+2)}-(q^{\delta+1}+1){\small\begin{bmatrix}
 2\delta +2 \\
 \delta +1\\
 \end{bmatrix}_q}\left(q^{2\delta(\delta+1)}+1\right).
\end{align*}
Then $g(\delta)\geq g_1(\delta)+g_2(\delta)$. Clearly
$$g_1(\delta)={\small\begin{bmatrix}
 2\delta \\
 \delta \\
 \end{bmatrix}_q}{\small\begin{bmatrix}
 2\delta +2 \\
 \delta +1\\
 \end{bmatrix}_q}\left(q^{3\delta+2}-q^{3\delta+1}-q^{2\delta}-q^{\delta}\right)\geq 0.$$
It remains to examine $g_2(\delta)\geq 0$.
Since
${\small\begin{bmatrix}
 2\delta +2 \\
 \delta +1\\
 \end{bmatrix}_q}=\frac{(q^{2\delta+2}-1)(q^{2\delta+1}-1)}{(q^{\delta+1}-1)^2}{\small\begin{bmatrix}
 2\delta \\
 \delta \\
 \end{bmatrix}_q}$ and
 $$q^{\delta}(q^{\delta+1}-1) \leq q^{2\delta+1}-1\leq (q^{\delta}+1)(q^{\delta+1}-1),$$
we have
$q^{\delta}(q^{\delta+1}+1){\small\begin{bmatrix}
 2\delta \\
 \delta \\
 \end{bmatrix}_q}\leq {\small\begin{bmatrix}
 2\delta +2 \\
 \delta +1\\
 \end{bmatrix}_q} \leq (q^{\delta}+1)(q^{\delta+1}+1){\small\begin{bmatrix}
 2\delta \\
 \delta \\
 \end{bmatrix}_q}$. Thus
\begin{align*}
g_2(\delta)\geq & q^{\delta}{\small\begin{bmatrix}
 2\delta \\
 \delta \\
 \end{bmatrix}_q}q^{2(\delta+1)(\delta+2)}-(q^{\delta+1}+1)(q^{\delta}+1)(q^{\delta+1}+1){\small\begin{bmatrix}
 2\delta \\
 \delta \\
 \end{bmatrix}_q}(q^{2\delta(\delta+1)}+1)\\
 =&{\small\begin{bmatrix}
 2\delta \\
 \delta \\
 \end{bmatrix}_q}(q^{2\delta^2+7\delta+4}-q^{2\delta^2+5\delta+2}-q^{2\delta^2+4\delta+2}-2q^{2\delta^2+4\delta+1}-2q^{2\delta^2+3\delta+1}-q^{2\delta^2+3\delta}
 \\&-q^{2\delta^2+2\delta}-q^{3\delta+2}-q^{2\delta+2}-2q^{2\delta+1}-2q^{\delta+1}-q^{\delta}-1)\geq 0.
\end{align*}
So given any prime power $q$, $f(\delta)$ is a non-increasing function on $\delta$ for any $\delta\geq 2$.

Since
\begin{align*}
f(2)&=\frac{(q^2+1){\small\begin{bmatrix}
 4 \\
 2 \\
 \end{bmatrix}_q}-q^4-q^2}{q^{12}+(q^4+q^2){\small\begin{bmatrix}
 4 \\
 2 \\
 \end{bmatrix}_q}+1}=\frac{(q^2+1)(q^4+q^3+q^2+q+1)}{q^{12}+(q^4+q^2)(q^2+1)(q^2+q+1)+1}\\
 &=\frac{q^6+q^5+2q^4+2q^3+2q^2+q+1}{q^{12}+q^8+q^7+3q^6+2q^5+3q^4+q^3+q^2+1}\leq\frac{155}{4797},
\end{align*}
we have for any $\delta\geq 2$,
\begin{align*}
\frac{{\rm the~lower~ bound~ of~}\mathcal C}{{\rm the~ upper~ bound~ of~}\mathcal C}&=1-f(\delta)\geq1-f(2)\geq 1-\frac{155}{4797}> 0.967688.
\end{align*}
Similarly since
\begin{align*}
f(3)&=\frac{(q^3+1){\small\begin{bmatrix}
 6 \\
 3 \\
 \end{bmatrix}_q}-q^6-q^3}{q^{24}+(q^6+q^3){\small\begin{bmatrix}
 6 \\
 3 \\
 \end{bmatrix}_q}+1}\leq \frac{12483}{16877657},
\end{align*}
we have for any $\delta\geq 3$,
\begin{align*}
\frac{{\rm the~lower~ bound~ of~}\mathcal C}{{\rm the~ upper~ bound~ of~}\mathcal C}=1-f(\delta)\geq 1-f(3)\geq 1-\frac{12483}{16877657}>0.99926.
\end{align*} \qed

\newpage

\section{Appendix}\label{app-B}

\begin{center} {Table~1: Constant dimension codes from Theorem \ref{new-3} and \cite{hkkw}\\
Lower bounds for ${\bar A}_q(n,2\delta,k)$} \label{T1} \vskip 2.5mm
\begin{tabular}{|c|c|c|}\hline
  ${\bar A}_q(n,2\delta,k)$ & Theorem \ref{new-3} & \cite{hkkw}  \\
  \hline
  ${\bar A}_3(15,6,6)$ &   150102543990846750 &     150102261281924288   \\
  \hline
    ${\bar A}_4(15,6,6)$ & 4722384778841908199452 & 4722384497336874172416  \\
  \hline
      ${\bar A}_5(15,6,6)$ & 14551922738557090682988320 & 14551922678951263378341888   \\
  \hline
      ${\bar A}_7(15,6,6)$ & 2651730911763599010817616 &  2651730911572017468075166 \\
      & 918746 &  138368  \\
  \hline
      ${\bar A}_8(15,6,6)$ & 3245185560810007534450532 &  3245185560762783660124143   \\
       & 05203320 & 69988608  \\
  \hline
      ${\bar A}_9(15,6,6)$ & 2252839960316867802912978 &  2252839960308891273979904  \\
              & 0303636252 & 2841640960  \\

  \hline

        ${\bar A}_3(16,6,6)$ & 12158306011246213950 &  12158283163835867136  \\
  \hline
             ${\bar A}_4(16,6,6)$ & 1208930503358636748324892 &   1208930431318239788138496  \\
  \hline
             ${\bar A}_5(16,6,6)$ & 909495171159508342705798  &   909495167434454027503612  \\
              & 8320 & 7232  \\
  \hline
            ${\bar A}_7(16,6,6)$ &  6366805919144396570740848 &   6366805918684414419355934  \\
             & 233189146 & 306336768  \\
  \hline
            ${\bar A}_8(16,6,6)$ &  1329228005707779000468281 &   1329228005688436187186849  \\
             & 480881260920 & 259473338368  \\
  \hline
            ${\bar A}_9(16,6,6)$ &  1478088297963896954275969  &  1478088297958663518971939  \\
            & 56677838811932 & 31748755374080  \\
 \hline
        ${\bar A}_3(17,6,6)$ &  984822786754900790910 &  984820936270705197056  \\
  \hline
          ${\bar A}_4(17,6,6)$ &  309486208859711440279978012 &   309486190417469385763454976  \\
  \hline

            ${\bar A}_5(17,6,6)$ & 56843448197469116506558079 &    56843447964653369704091609  \\
              & 88320 &  98912 \\
  \hline
            ${\bar A}_7(17,6,6)$ & 15286701011865696133769150  &   15286701010761278864076273  \\
            & 164214433946 &  642983391232  \\
  \hline
            ${\bar A}_8(17,6,6)$ & 54445179113790627852329396  &   54445179112998346227173345 \\
            & 89634331511160 &  66802793955328   \\
  \hline
            ${\bar A}_9(17,6,6)$ & 96977373229411279169036992 &    96977373229067909497849520  \\
            & 0953064514284572 &  2646735567454208   \\
  \hline
        ${\bar A}_3(17,6,7)$ &      717934513945606807214448  &     717934462541344066764800  \\
  \hline
          ${\bar A}_4(17,6,7)$ &    1267655437138714999914722 &     1267655435949954604087111  \\
          & 735680 &  581696  \\
  \hline

            ${\bar A}_5(17,6,7)$ &  8881788744768854329316669 &     8881788744477090184835827  \\
            & 4301204500 &  3727594496  \\
  \hline
            ${\bar A}_7(17,6,7)$ &  1798465087215432610771286  &    1798465087215053471877135 \\
            & 951668456912186048 &  499494277221711872  \\
  \hline
            ${\bar A}_8(17,6,7)$ &  1427247703339824504317286  &    1427247703339783847337612 \\
             & 302786023739979637248 &  952679951618625503232  \\
  \hline
      ${\bar A}_9(17,6,7)$ &        5153775220622933028375569 &     5153775220622908412230038  \\
           & 74561316380272313825860 &  58973884622428476276736  \\

  \hline
  ${\bar A}_3(18,6,7)$ & 174458086133950569694999932 & 174458074397546630638206976   \\
  \hline
\end{tabular}
 \end{center}

\begin{center}{Table~1 (Cont.): Constant dimension codes from Theorem \ref{new-3} and \cite{hkkw}\\
Lower bounds for ${\bar A}_q(n,2\delta,k)$} \label{T1-1} \vskip 2.5mm
\begin{tabular}{|c|c|c|}
  \hline
  ${\bar A}_q(n,2\delta,k)$ & Theorem \ref{new-3} & \cite{hkkw}  \\

  \hline
 ${\bar A}_4(18,6,7)$ &  129807916760321574216291241 & 129807916641275351458520225   \\
  & 3988416 & 9656704  \\
  \hline
 ${\bar A}_5(18,6,7)$ &  277555898273931997862551572 & 277555898264909066117850548  \\
  & 409926204500 & 774715260928  \\
  \hline
   ${\bar A}_7(18,6,7)$ &302268027208297537838299541 & 302268027208234077195734138   \\
     & 69875398154049929084 & 40701999450419101696  \\
  \hline
   ${\bar A}_8(18,6,7)$ &467680527430393663434529587 & 467680527430380371095589012   \\
     & 83305087356415231304192 & 33416654639120489906176  \\
  \hline
   ${\bar A}_9(18,6,7)$ &304325273002563570083080541 & 304325273002562128138529778  \\
   & 77589445454500732900196836 & 45672945438783195662254080  \\
  \hline
 ${\bar A}_3(19,6,7)$ &  42393314923753431509633199 &  42393312078603826847037784   \\
  & 312 & 064  \\
  \hline
 ${\bar A}_4(19,6,7)$ &  132923306762526363919145683 & 132923306640665959893524711   \\
  & 3419327040 & 3888464896  \\
  \hline
 ${\bar A}_5(19,6,7)$ &  867362182106035125792002834 & 867362182077840679822534739  \\
 & 689855238704500 & 559109441355776  \\
  \hline
   ${\bar A}_7(19,6,7)$ &508021873328985670761663974 & 508021873328878943122344222   \\
   & 813537143982278408961152 & 542078792818337377157120  \\
  \hline
   ${\bar A}_8(19,6,7)$ &153249555228391395614936944 &   153249555228387040000602607   \\
      & 5210953425471736128141435392 & 5616596939214700213245575168  \\
  \hline
   ${\bar A}_9(19,6,7)$ &17970103045528376249648755925 & 17970103045528289061322277127  \\
       & 65293397244462194049577733380 & 01292762916333299089237082112  \\
  \hline
\end{tabular}
 \end{center}

\newpage

\begin{center} {Table~2: Constant dimension codes from Theorem \ref{new-3}, \ref{con4}, \ref{cdc45} and \cite{hkkw}\\
Lower bounds for ${\bar A}_q(n,2\delta,k)$}\vskip 2.5mm
\begin{tabular}{|c|c|c|c|c|}
  \hline
${\bar A}_q(n,2\delta,k)$   & Theorem \ref{cdc45}  &Theorem \ref{con4}& Theorem \ref{new-3}&  \cite{hkkw} \\
   \hline
  ${\bar A}_3(13,4,5)$ &           187977330080&187644030023&187623140212&187646890063\\
  &            0662 &1043&3284&  3708\\
  \hline
     ${\bar A}_4(13,4,5)$ &        185244551321&185193502027&185191080330&185193668253 \\
       & 42240085 &  94936427 & 51471424  &  31922597  \\
  \hline
      ${\bar A}_5(13,4,5)$ &       233220333417&233204343725&233203799712&233204366687 \\
        & 60498047656 &  35344411636 &  28104104500  &  01425801556  \\
  \hline
       ${\bar A}_7(13,4,5)$ &      110489772145&110488804026&110488785701&110488804410 \\
                        &   826259551989&935735689413&592348093173&257915997937  \\
                         &4274 & 6924  &  0404  &  0508  \\
  \hline
         ${\bar A}_8(13,4,5)$ &    792478237928&792475147017&792475101353&792475147744  \\
                           &289801654459&480054665841&872707790607&175434285903  \\
                           &47977 & 57579  & 24224   & 14633  \\
  \hline
           ${\bar A}_9(13,4,5)$ &  343421389872&343420732650&343420724881&343420732748 \\
                           &099066234534&983522582815&350622237717&274503618449\\
                           &2116660&7580092&6247844 &0814932\\
  \hline
  ${\bar A}_3(14,4,5)$ &    152290128114     &151991664671&151971000771&151993980961 \\
                            & 549994 &267113&313764&  035277 \\
  \hline
     ${\bar A}_4(14,4,5)$ &        474234835438&474095365198&474088568287&474095790728 \\
                           &9224194389 & 0418663364  &  3806466624  &  5780822949  \\
  \hline
      ${\bar A}_5(14,4,5)$ &       145763149391&145752714828&145752356230&145752729179\\
                           &911682128914&382638784555&952762603545&384155801334 \\
                           &06&29  & 00  & 31 \\
  \hline
       ${\bar A}_7(14,4,5)$ &      265285995217&265283618468&265283573374&265283619389\\
                           &197106992212&672861039951&849294506738&029256718151 \\
                           &2343434&4145353&6763204& 6018577 \\
  \hline
         ${\bar A}_8(14,4,5)$ &    324599107938&324597820218&324597801169&324597820516  \\
                          & 548775888100&359848132646&235252976985&014257914299 \\
                          & 448539209& 389432068&412575744& 810657417 \\

  \hline
           ${\bar A}_9(14,4,5)$ &  225318779695& 225318342692&   225318337521 &     225318342756 \\
                           &316687971945 &310290649450&979435078382 &  142901825894\\
            &70848710934 &77022477769 &85688931844& 86571731339\\
 \hline
\end{tabular}
 \end{center}

\newpage

\begin{center}
{Table~3: Constant dimension codes from Construction \ref{new} and \cite{hkkw}\\
Lower bounds for ${\bar A}_q(n,2\delta,k)$}\vskip 2.5mm
\begin{tabular}{|c|c|c|}
  \hline
  ${\bar A}_q(n,2\delta,k)$ & Construction \ref{new} & \cite{hkkw}  \\
  \hline
   ${\bar A}_2(12,6,6)$ & 16865174 & 16865101  \\
  \hline
  ${\bar A}_3(12,6,6)$ & 282454201878 & 282454201122   \\
  \hline
    ${\bar A}_4(12,6,6)$ & 281476519731292 & 281476519727132  \\
  \hline
      ${\bar A}_5(12,6,6)$ & 59604684750285320 & 59604684750269570   \\
  \hline
      ${\bar A}_7(12,6,6)$ & 191581237048517757994 & 191581237048517640002   \\
  \hline
      ${\bar A}_8(12,6,6)$ & 4722366523787007642488 & 4722366523787007379832   \\
  \hline
      ${\bar A}_9(12,6,6)$ & 79766443311676870585932 & 79766443311676870053762  \\
  \hline
       ${\bar A}_2(14,6,7)$ & 34532242376 & 34532238023  \\
  \hline
        ${\bar A}_3(14,6,7)$ &  50035894106925204 &  50035894106387202  \\
  \hline
          ${\bar A}_4(14,6,7)$ &  1180598085852258350656 &  1180598085852241507904  \\
  \hline

            ${\bar A}_5(14,6,7)$ &  2910384996920980879329500 &  2910384996920980634798250  \\
  \hline
            ${\bar A}_7(14,6,7)$ &  378818703472375564731912769036  &  378818703472375564718065717034  \\
  \hline
            ${\bar A}_8(14,6,7)$ &  40564819558769908757756030657 &  40564819558769908757687294403  \\
            & 024 &   072  \\
  \hline

  ${\bar A}_9(14,6,7)$ &  25031555123615248786076588088 &  25031555123615248786073763362  \\
  & 37716 &   54514  \\
  \hline

  ${\bar A}_2(16,8,8)$ &  1099562832574 &  1099562828461  \\
  \hline
   ${\bar A}_3(16,8,8)$ &  12157665957048196644 &  12157665957047665122  \\
  \hline
   ${\bar A}_4(16,8,8)$ &  1208925820022362634893084 &  1208925820022362618115612  \\
  \hline
   ${\bar A}_5(16,8,8)$ &  9094947017807612368246590820  &  9094947017807612368002449570  \\
  \hline
   ${\bar A}_7(16,8,8)$ &  63668057609092569002476703621 &  63668057609092569002476565208  \\
             & 23204 &   33602  \\
  \hline
   ${\bar A}_8(16,8,8)$ &  13292279957849213674394204065 &  13292279957849213674394203378  \\
     & 92780664 &   73299832  \\
  \hline
   ${\bar A}_9(16,8,8)$ & 14780882941434601431198749325 & 14780882941434601431198749296\\

  & 1158647364 &   8729104322  \\
  \hline

${\bar A}_2(16,6,8)$ &  282927684131264 &  282927683836351  \\
  \hline
  ${\bar A}_3(16,6,8)$ &  79773403858211769073398 &  79773403858211367304002  \\
  \hline
  ${\bar A}_4(16,6,8)$ &  79228596795209597355803963392 &  79228596795209597286010744832  \\
  \hline
  ${\bar A}_5(16,6,8)$ &  3552716061446350478567982091 &  3552716061446350478564136876  \\
   & 625000 &   781250  \\
  \hline
  ${\bar A}_7(16,6,8)$ &  3670336930316550640268162626  &  3670336930316550640268162462  \\
   & 0312448748394 &   7151289328002  \\
  \hline
  ${\bar A}_8(16,6,8)$ &  2230074539175728767236156261 &   2230074539175728767236156259  \\
  & 8047925701050368 &   9998342819479552  \\
  \hline
  ${\bar A}_9(16,6,8)$ &  6362685459865446204861526038 &  63626854598654462048615260385 \\
   & 705059941329515532 &   54759414900421762  \\
  \hline
 ${\bar A}_2(18,8,9)$ & 18015215398134856 &  18015215398068295  \\
  \hline
  ${\bar A}_3(18,8,9)$ &   58149739380417667241629716 &  58149739380417667198523946  \\
  \hline
${\bar A}_4(18,8,9)$ &  32451855376784298642321718 &  32451855376784298642321288  \\
     & 2266944 &   6251072  \\
  \hline
   ${\bar A}_5(18,8,9)$ &  55511151231735878357960116 &  55511151231735878357960116 \\
    & 981761704500 &   829164048250  \\
    \hline
\end{tabular}
 \end{center}

\begin{center}{Table~3 (Cont.): Constant dimension codes from Construction \ref{new} and \cite{hkkw}\\
Lower bounds for ${\bar A}_q(n,2\delta,k)$}\vskip 2.5mm
\begin{tabular}{|c|c|c|}
    \hline
  ${\bar A}_q(n,2\delta,k)$ & Construction \ref{new} & \cite{hkkw}  \\
  \hline

   ${\bar A}_7(18,8,9)$ &  43181145673965918176230160  & 43181145673965918176230160  \\
     & 95285332497749370596 &  95285299264536325746  \\
  \hline
   ${\bar A}_8(18,8,9)$ &  58460065493236358379340343 &  58460065493236358379340343  \\
      & 02923933871658428965376 &  02923933590182378512896  \\
  \hline
   ${\bar A}_9(18,8,9)$ &  33813919135227284246202802 &  33813919135227284246202802  \\
       & 47018514715266280331502884 &  47018514713413256655866642  \\
  \hline

   ${\bar A}_2(18,6,9)$ &  9271545156585415680 &  9271545156551861247  \\
  \hline
   ${\bar A}_3(18,6,9)$ &  11446612801881132293137038  &  11446612801881132287488447  \\
    & 59802 &  86840  \\
  \hline
   ${\bar A}_4(18,6,9)$ &  85071058146182803276503914 &   85071058146182803276503351  \\
    & 069802090496 &  119848669184  \\
  \hline
   ${\bar A}_5(18,6,9)$ &  10842028996571097790669084 &  108420289965710977906690845  \\
     & 5136306309921875000 &  017097020371093750  \\
  \hline
   ${\bar A}_7(18,6,9)$ &  174251503388975551318884922 &  174251503388975551318884922  \\
    & 599369849935281754612330830 &  599369466772818993479502028  \\
  \hline
   ${\bar A}_8(18,6,9)$ &  78463772372191979113838163463 & 78463772372191979113838163463  \\
      & 5235743275201736965558894592 &  5235733830468771226268467200  \\
  \hline
   ${\bar A}_9(18,6,9)$ &  1310020512493866339206870302329 &  1310020512493866339206870302329  \\
        & 188713507904303585133560754636 &  188713348371417431388541027914  \\
  \hline
\end{tabular}

 \end{center}

\end{document}